\documentclass[11pt]{article}
\usepackage{amsthm, amsmath, amssymb, amsfonts, url, booktabs, tikz, setspace, fancyhdr, bm}
\usepackage{hyperref}
\usepackage{amsthm}
\usepackage{geometry}
\usepackage{cancel}
\usepackage{hyperref, enumerate}
\usepackage[shortlabels]{enumitem}
\usepackage[babel]{microtype}
\usepackage[english]{babel}
\usepackage[capitalise]{cleveref}
\usepackage{comment}
\usepackage{bbm}
\usepackage{csquotes}
\usepackage{mathabx}
\usepackage{tikz}
\usepackage{graphicx}
\usepackage{float}
\usepackage{xcolor}
\usetikzlibrary{positioning, arrows.meta, shapes.geometric}
\usepackage{booktabs}

\counterwithin{figure}{section}

\newtheorem{theorem}{Theorem}[section]
\newtheorem{prop}[theorem]{Proposition}
\newtheorem{lemma}[theorem]{Lemma}
\newtheorem{cor}[theorem]{Corollary}

\newtheorem{claim}[theorem]{Claim}

\definecolor{rosepink}{RGB}{255,102,204}
\definecolor{dateplum}{HTML}{993366}
\definecolor{darkdateplum}{RGB}{128,0,32}
\definecolor{lightdateplum}{RGB}{219,112,147}
\definecolor{darkred}{RGB}{139,0,0}
\definecolor{lightred}{RGB}{240,130,100}
\theoremstyle{definition}
\newtheorem{defn}[theorem]{Definition}
\newtheorem*{defn-non}{Definition}

\newtheorem{constr}[theorem]{Construction}

\newcommand{\Yemph}[1]{\textcolor{black}{\emph{#1}}}

\newtheorem{rmk}[theorem]{Remark}

\newlist{Case}{enumerate}{2}
\setlist[Case, 1]{%
    label           =   {\bfseries Case \arabic*.},
    labelindent=1em ,labelwidth=1.3cm, labelsep*=1em, leftmargin =!
}
\setlist[Case, 2]{%
    label           =   {\bfseries Subcase \arabic{Casei}.\arabic*.},
    labelindent=-1em ,labelwidth=1.3cm, labelsep*=1em, leftmargin =!
}

\newenvironment{poc}{\begin{proof}[Proof of claim]}{\end{proof}}

\newcommand{\C}[1]{{\protect\mathcal{#1}}}

\newcommand{\I}[1]{{\mathbbm #1}}

\newcommand{\M}[1]{\mathrm{#1}}

\usepackage{todonotes} 

\newcommand{\eps}{\varepsilon}

\newcommand{\blue}{\textcolor{blue}}

\newcommand{\match}{\nu_\mathrm{bi}}

\newcommand{\VC}{\mathrm{dim}_{\mathsf{VC}}}
\newcommand{\MIS}{\mathrm{MIS}}

\title{Beyond chromatic threshold via the $(p,q)$-theorem, and a sharp blow-up phenomenon} 
\author{
Hong Liu\thanks{Extremal Combinatorics and Probability Group (ECOPRO), Institute for Basic Science (IBS), Daejeon, South Korea. Emails: {\texttt \{hongliu, zixiangxu\}@ibs.re.kr}. Supported by IBS-R029-C4.}
\and
Chong Shangguan\thanks{Research Center for Mathematics and Interdisciplinary Sciences, Shandong University, Qingdao 266237, China, and Frontiers Science Center for Nonlinear Expectations, Ministry of Education, Qingdao 266237, China. Email: {\texttt theoreming@163.com}. Supported by National Natural Science Foundation of China under Grant Nos. 12101364 and 12231014, and Natural Science Foundation of Shandong Province under Grant No. ZR2021QA005.}
\and
Jozef Skokan\thanks{Department of Mathematics, London School of Economics, Houghton Street, London WC2A 2AE, UK.
Email: \texttt{j.skokan@lse.ac.uk}}
\and
Zixiang Xu\footnotemark[1]
}

\begin{document}

\maketitle
\begin{abstract}
We establish a novel connection between the well-known chromatic threshold problem in extremal combinatorics and the celebrated $(p,q)$-theorem in discrete geometry. In particular, for a graph $G$ with bounded clique number and a natural density condition, we prove a $(p,q)$-theorem for an abstract convexity space associated with $G$. Our result strengthens those of Thomassen and Nikiforov on the chromatic threshold of cliques. Our $(p,q)$-theorem can also be viewed as a $\chi$-boundedness result for (what we call) ultra maximal $K_r$-free graphs. 

We further show that the graphs under study are blow-ups of constant size graphs, improving a result of Oberkampf and Schacht on homomorphism threshold of cliques. Our result unravels the cause underpinning such a blow-up phenomenon, differentiating the chromatic and homomorphism threshold problems for cliques. It implies that for the homomorphism threshold problem, rather than the minimum degree condition usually considered in the literature, the decisive factor is a clique density condition on co-neighborhoods of vertices. More precisely, we show that if an $n$-vertex $K_{r}$-free graph $G$ satisfies that the common neighborhood of every pair of non-adjacent vertices induces a subgraph with $K_{r-2}$-density at least $\eps>0$, then $G$ must be a blow-up of some $K_r$-free graph $F$ on at most $2^{O(\frac{r}{\eps}\log\frac{1}{\eps})}$ vertices. Furthermore, this single exponential bound is optimal. We construct examples with no $K_r$-free homomorphic image of size smaller than $2^{\Omega_r(\frac{1}{\varepsilon})}$.
\end{abstract}

\section{Introduction}
\subsection{Overview}
Finding sufficient conditions guaranteeing a graph to have bounded complexity has long been a popular topic in combinatorics and theoretical computer science. There are many natural ways to measure complexity. In this paper, the invariant we are interested in is the chromatic number. Obviously the chromatic number of a graph is lower bounded by its clique number and, therefore, we focus on graphs with bounded clique number in this paper. On the other hand, there are plethora of graphs with bounded clique number and arbitrarily large chromatic number, such as Mycielski graphs, an important family of triangle-free graphs in structural graph theory, Kneser graphs, a well-studied family in topological and algebraic combinatorics, and, as one of the famous early applications of the probabilistic method in combinatorics, Erd\H{o}s~\cite{1959Erdos} in 1959 gave a random construction with arbitrarily large girth (hence without triangles) and chromatic number. 

A natural question which attracted a lot of attention is whether the chromatic number can be bounded if an additional condition is imposed. For example, 
Tur\'{a}n's theorem~\cite{1941Turan}, a fundamental result in extremal graph theory, states that every $n$-vertex $K_{r}$-free graph has average degree at most $\frac{r-2}{r-1}\cdot n$, and, moreover, the balanced complete $(r-1)$-partite graph, also known as the Tur\'{a}n graph $T_{n,r-1}$, is the unique extremal graph. 
Rephrasing Tur\'{a}n's theorem, we see that every $n$-vertex $K_{r}$-free graph $G$ with average degree at least $\frac{r-2}{r-1}\cdot n$ must be $T_{n,r-1}$ and hence $(r-1)$-colorable. However, the average degree condition is not an ideal condition to force bounded chromatic number due to small `local noise'. Indeed, if the average degree drops slightly, although by a stability theorem of Erd\H{o}s and Simonovits~\cite{1966ErdosStability,1966SimonovitsStability}, we know that $G$ is close to $T_{n,r-1}$ (in edit-distance), but $G$ could be e.g. a disjoint union of $T_{n-\sqrt{n},r-1}$ and an Erd\H{o}s's random graph on $\sqrt{n}$ vertices with large girth and chromatic number. Notice that here the vertices in the `local noise' part have $o(n)$ small degree. It is then perhaps not too surprising that the minimum degree condition is the correct one to impose: Andr\'{a}sfai, Erd\H{o}s, and S\'{o}s~\cite{1974ErdosSos} showed that every $n$-vertex $K_{r}$-free graph with minimum degree larger than $\frac{3r-7}{3r-4}\cdot n$ has chromatic number at most $r-1$. It remained an interesting question to determine the optimal min-degree condition guaranteeing bounded chromatic number. This is the by-now well-known 
\Yemph{chromatic threshold} problem first formulated in the early 1970s by Erd\H{o}s and Simonovits~\cite{1973ErdosSimonovits}, and since then there has been a large amount of work on this topic. Formally, for a graph $H$, its chromatic threshold is defined as
$$\delta_{\chi}(H):=\inf\big\{\alpha\ge 0:\exists~C=C(\alpha,H)\  \textup{s.t.~} \forall~n\textup{-vertex\ }H\textup{-free}~G,~\delta(G)\ge \alpha n\Rightarrow \chi(G)\le C \big\}.$$
In other words, for $\alpha<\delta_{\chi}(H)$, there exist $H$-free graphs with minimum degree $\alpha n$ and arbitrarily large chromatic number, but if $\alpha>\delta_{\chi}(H)$, then $\chi(G)$ must be bounded. 

For the first non-trivial case when $H$ is a triangle, a beautiful construction of Hajnal using the Kneser graph shows that $\delta_{\chi}(K_3)\ge \frac{1}{3}$ (see~\cite{1973ErdosSimonovits}); this example can be easily extended to obtain $\delta_{\chi}(K_r)\ge \frac{2r-5}{2r-3}$. It was not until 2002 that Thomassen~\cite{2002Thomassen} proved that Hajnal's construction is optimal: $\delta_{\chi}(K_3)= \frac{1}{3}$, as conjectured by Erd\H{o}s and Simonovits~\cite{1973ErdosSimonovits}. Later, the chromatic thresholds for all cliques $\delta_{\chi}(K_{r})=\frac{2r-5}{2r-3}$ were determined by Goddard and Lyle~\cite{2011JGTKrChromatic} and independently by Nikiforov~\cite{2010arxivKrfree}. After a series of results (see, e.g.,~\cite{1999CPCBrandt,2011Unpubilished,1997CPCChen,1982OddCycles,1995DMJin, 2010ColoringViaVCDim}), the culmination is the remarkable work of Allen, B\"{o}ttcher, Griffiths, Kohayakawa, and Morris~\cite{2013advAllChromatic}, which determined the chromatic thresholds for all graphs $H$. 

In this paper, we investigate a different density condition: a clique density condition in co-neighborhoods of non-adjacent vertices.
Below is our first result.

\begin{theorem}\label{thm:main0}
 Let $r\ge 3$, $\varepsilon>0$ and $G$ be an $n$-vertex $K_{r}$-free graph. If for every non-adjacent pair of vertices $u,v\in V(G)$, $G[N(u)\cap N(v)]$, the subgraph of $G$ induced on the co-neighborhood $N(u)\cap N(v)$ of $u$ and $v$, contains at least $\varepsilon n^{r-2}$ copies of $K_{r-2}$, then $\chi(G)=O_{r,\varepsilon}(1)$\footnote{Here we write $O_{r,\varepsilon}(1)$ for a constant that depends only on $r$ and $\varepsilon$.}. 
\end{theorem}

Our theorem implies and extends the results of Thomassen~\cite{2002Thomassen}, Goddard and Lyle~\cite{2011JGTKrChromatic}, and Nikiforov~\cite{2010arxivKrfree} as the above clique density condition in the co-neighborhoods is a strictly weaker condition than the minimum degree condition in $\delta_{\chi}(K_r)$; see the discussion in~\cref{sec:application}. Throughout the rest of this paper, we will refer the graphs in~\cref{thm:main0} as \Yemph{$\varepsilon$-ultra maximal $K_{r}$-free graphs}.

One of our main contributions is to study the chromatic threshold problem through a geometric perspective, establishing a surprising connection between this classical problem in extremal combinatorics and the celebrated $(p,q)$-theorem in discrete geometry. We will elaborate more on this in the next subsection.

\subsection{Chromatic threshold meets the $(p,q)$-theorem}
In 1921, Radon~\cite{1921Radon} proved his fundamental lemma in combinatorial convexity which states that any set of $d+2$ points in $\I R^d$ can be partitioned into two parts whose convex hulls intersect. Radon's lemma was introduced to prove Helly’s theorem~\cite{1923Helly}, yet another fundamental result in convexity. Helly's theorem is a local-global type result, stating that for any finite family of convex sets in $\I R^d$, if every $d+1$ of them intersect, then all of them intersect. Since then, many generalizations have been developed. We refer the readers to~\cite{2019Survey,2017BU,2019HellySurvey,1993HellySurvey,2017DMMinki,2024sigmaMinki}. One of the important extensions, due to Katchalski and Liu~\cite{1979PAMS}, is the fractional Helly theorem. It states that if in a finite family of convex sets in $\I R^d$, a positive fraction of $(d+1)$-tuples intersect, then a positive fraction of all sets intersect. For $p\ge q$, a family of sets has the \Yemph{$(p,q)$-property} if each of its $p$-tuples contains an intersecting $q$-tuple. In the early 90s, settling an old conjecture of Hadwiger and Debrunner~\cite{1957HadwigerDeb}, Alon and Kleitman~\cite{1992ALonKleitman} proved the following famous $(p,q)$-theorem, a far-reaching generalization of Helly's theorem. For a set system $\mathcal{G}$ defined on a ground set $X$, a set $T\subseteq X$ is a \Yemph{transversal} for $\mathcal{G}$ if for every $A\in \C G$, we have $T\cap A\neq\varnothing$. The \Yemph{transversal number} of $\mathcal{G}$, denoted by $\tau(\mathcal{G})$, is the minimum size of a transversal for $\mathcal{G}$.

\begin{theorem}[$(p,q)$-theorem~\cite{1992ALonKleitman}]
    Let $d,p,q$ be positive integers with $p\ge q\ge d+1$. If $\mathcal{F}$ is a finite family of convex sets in $\I R^{d}$ with the $(p,q)$-property, then $\tau(\mathcal{F})\le O_{d,p,q}(1)$.
\end{theorem}

An interesting direction of research in discrete geometry is to prove classical results in Euclidean convexity in an abstract way. We will also work in an axiomatic setting through abstract convexity spaces. By now, abstract versions of Helly's theorem and many of its variants, including the fractional Helly theorem, and $(p,q)$-theorem have been studied in abstract convexity spaces, see e.g.~\cite{2017HellyNumber,2021IJMHolmsen,2020DCGWeakNetRadon}. 

We will cast the problem in~\cref{thm:main0} in the language of convexity spaces, which turns out to be equivalent to the following.

\begin{theorem}\label{thm:main-rephased}
    Let $r\ge 3$, $\eps>0$ and $\C B$ be a set system with the $(r,2)$-property. If for every intersecting pair $A,B\in\C B$, at least an $\eps$ fraction of the $(r-2)$-tuples $\C T$ in $\C B$ satisfies that both $\C T\cup \{A\}$ and $\C T\cup \{B\}$ are pairwise disjoint $(r-1)$-tuples, then $\tau(\C B)= O_{r,\eps}(1)$.
\end{theorem}

We further strengthen~\cref{thm:main0,thm:main-rephased} to the following, which is a $(p,q)$-theorem for certain convexity spaces we construct from graphs. We defer the definition of the convexity space $\C C(G)$ below and other necessary notions to~\cref{sec:GeometricConsequences}. We refer readers familiar with convexity spaces to~\cref{sec:pf-map} and~\cref{fig:Relation} which contain an overview of the proof for~\cref{thm:pq-for-C}.

\begin{theorem}[$(r,2)$-theorem for $\C C(G)$]\label{thm:pq-for-C}
  Let $r\ge 3$ and $G$ be a graph with bipartite induced matching number $t$. For any $\C G\subseteq \C C(G)$, if $\C G$ has the $(r,2)$-property, then $\tau(\C G)= O_{r,t}(1)$. In particular, if $G$ is $\alpha$-ultra maximal $K_r$-free, then $\tau(\C G)= O_{r,\alpha}(1)$.
\end{theorem}

We remark that~\cref{thm:pq-for-C} is also related to the well-studied $\chi$-boundedness problem~\cite{2020ScottSurvey}. Indeed, it can be viewed also as a strengthening of a $\chi$-boundedness result for ultra maximal $K_r$-free graphs, see~\cref{rmk:chi-bounded}.

\subsection{Blow-up phenomenon}
Given two graphs $G$ and $F$, we say that $G$ is \Yemph{homomorphic} to $F$, denoted by $G\xrightarrow{\textup{hom}} F$, if there exists a homomorphism $\varphi:V(G)\rightarrow V(F)$ preserving adjacencies, i.e.~if $uv\in E(G)$, then $\varphi(u)\varphi(v)\in E(F)$. Note that $\chi(G)=t$ is equivalent to $G$ being homomorphic to $K_{t}$. A natural extension of the chromatic threshold problem raised in~\cite{2002Thomassen} asks for minimum degree condition for an $H$-free graph guaranteeing a bounded size homomorphic image which is also $H$-free. Formally, the \Yemph{homomorphism threshold} of a graph $H$ is defined as
$$\delta_{\textup{hom}}(H):=\inf\big\{\alpha\ge 0:\exists~H\textup{-free}~F=F(\alpha,H)\  \textup{s.t.~}\forall~n\textup{-vertex\ }H\textup{-free}~G,~\delta(G)\ge \alpha n\Rightarrow G\xrightarrow{\textup{hom}} F  \big\}.$$    
As having a bounded homomorphic image implies bounded chromatic number, we have $\delta_{\textup{hom}}(H)\ge\delta_{\chi}(H)$. The first homomorphism threshold was established by {\L}uczak~\cite{2006CombTriangleHom} who showed that $\delta_{\textup{hom}}(K_{3})=\frac{1}{3}$. Later, Goddard and Lyle~\cite{2011JGTKrChromatic} resolved the clique case, proving that $\delta_{\textup{hom}}(K_{r})=\delta_{\chi}(K_r)=\frac{2r-5}{2r-3}$ for all positive integers $r\ge 3$. The proofs of {\L}uczak~\cite{2006CombTriangleHom}, and Goddard and Lyle~\cite{2011JGTKrChromatic} utilized Szemer\'{e}di's regularity lemma~\cite{1978OriginalRegularity}, and therefore gave a tower-type upper bound on the size of the $K_r$-free homomorphic image $F$. Recently, Oberkampf and Schacht~\cite{2020CPCProb} gave a new proof using a clever probabilistic argument. Their proof yields a double exponential bound $2^{2^{O(1/\varepsilon^{2})}}$ on the size of the homomorphic image for $K_r$-free graphs with minimum degree at least $(\frac{2r-5}{2r-3}+\eps)n$, which was the best-known bound for all $r\ge 4$. As for $r=3$, a beautiful (unpublished) result of Brandt and Thomass\'{e}~\cite{2011Unpubilished} proved an optimal bound that any triangle-free graph $G$ with $\delta(G)\ge(\frac{1}{3}+\eps)n$ is homomorphic to a triangle-free graph (so-called Vega graph) of size $O\big(\frac{1}{\varepsilon}\big)$.

Except for cliques, determining the homomorphism threshold for any other graph is still widely open. Letzter and Snyder~\cite{2019JGTC3C5} showed that $\delta_{\textup{hom}}(C_{5})\le\frac{1}{5}$. Later, Ebsen and Schacht~\cite{2020HomoOddCycles} proved that $\delta_{\textup{hom}}(C_{2r+1})\le\frac{1}{2r+1}$ for every $r\ge 2$. A recent significant advancement, due to Sankar~\cite{2022Maya}, shows that $\delta_{\textup{hom}}(C_{2r+1})>0$. It was known that $\delta_{\chi}(C_{2r+1})=0$~\cite{2007OddCycleChromatic}, so her result provides the first example $H$ with $\delta_{\textup{hom}}(H)>\delta_{\chi}(H)$. 

An important variation was studied by {\L}uczak and Thomass{\'e}~\cite{2010ColoringViaVCDim} in their ingenious new proof of $\delta_{\chi}(K_{3})=\frac{1}{3}$. They realized that the $n/3$ minimum degree condition can be relaxed to a linear one if we have bounded VC-dimension (see~\cref{sec:conv-sp} for definition). That is, any $n$-vertex triangle-free graph $G$ with bounded VC-dimension and $\delta(G)=\Omega(n)$ has bounded chromatic number (see~\cref{thm:VCToBoundedChromaticNumber}). 

Our next result shows that VC-dimension, while significant in the chromatic threshold problem, is not so influential in the homomorphism threshold problem. 

\begin{theorem}\label{thm:VC-notfor-hom}
    There exists an $n$-vertex triangle-free graph $G$ with VC-dimension $3$ and $\delta(G)\ge \frac{n}{4}$ such that $G$ has no triangle-free homomorphic image of size smaller than $\frac{n}{4}$.
\end{theorem}

What is more relevant then? This is the content of our next main result, which strengthens~\cref{thm:main0}. Our theorem shows that the graph under study is in fact a blow-up of a constant size graph, which then necessarily also has to be $K_r$-free. For a graph $H$, we write $H[t]$ for the \Yemph{$t$-blow-up} of $H$ obtained by replacing every vertex of $H$ by $t$ independent copies (i.e. every vertex becomes an independent set of size $t$ and every edge becomes a copy of $K_{t,t}$). We simply write $H[\cdot]$ when mentioning a blow-up of $H$ without specifying its size (and the sizes of these independent sets could be different). Obviously, if $G=F[\cdot]$, then $G\xrightarrow{\textup{hom}} F$.

\begin{theorem}\label{thm:main1}
    Given $r\ge 3, \eps>0$ and $G$ be an $n$-vertex $K_{r}$-free graph. If for every non-adjacent pair of vertices $u,v\in V(G)$, the induced subgraph $G[N(u)\cap N(v)]$ contains $\varepsilon n^{r-2}$ copies of $K_{r-2}$, then $G=F[\cdot]$ for some maximal $K_{r}$-free graph $F$ on at most $2^{O((\frac{1}{\eps}+r)\log\frac{1}{\eps})}$ vertices. 
\end{theorem}

The merit of~\cref{thm:main1} is that it reveals the cause of the blow-up phenomenon $G=F[\cdot]$. Indeed, since the chromatic and homomorphism thresholds mysteriously coincide for cliques: $\delta_{\textup{hom}}(K_{r})=\delta_{\chi}(K_r)$, it was not clear what separates these two problems for cliques. \cref{thm:main1} shows that for the blow-up phenomenon, rather than VC-dimension, the main driving factor is the clique density condition in the co-neighborhoods considered above.

Interestingly, in contrast to the $O(\frac{1}{\varepsilon})$ bound on the size of triangle-free homomorphic image in Brandt-Thomass\'{e}'s work~\cite{2011Unpubilished} under the stronger minimum degree condition, much to our own surprise, the single exponential bound in $\frac{1}{\eps}$ in~\cref{thm:main1} is optimal as shown by the following construction.

\begin{theorem}\label{thm:main1-lowerbound}
    For every $r\ge 3$, there exists an $n$-vertex $K_{r}$-free graph $G$ such that for every pair of non-adjacent vertices $u,v$, $G[N(u)\cap N(v)]$ contains at least $\varepsilon n^{r-2}$ copies of $K_{r-2}$, but $G$ has no $K_{r}$-free homomorphic image of size smaller than $2^{\frac{1}{8r^r\varepsilon}}$.
\end{theorem}

\subsection{Applications}\label{sec:application}
As a first application, the chromatic and homomorphism thresholds for cliques, $\delta_{\textup{hom}}(K_{r})=\delta_{\chi}(K_r)=\frac{2r-5}{2r-3}$, are corollaries of~\cref{thm:main0,thm:main1}, and~\cref{thm:main1} quantitatively improves the double exponential bound in~\cite{2020CPCProb} to a single exponential one. Indeed, the minimum degree condition $\delta(G)\ge(\frac{2r-5}{2r-3}+\varepsilon)n$ implies the clique density condition in co-neighborhoods considered in~\cref{thm:main0,thm:main1}, see~\cref{prop:MinDegreeToUltra}. 

In addition, one can consider the following natural variations on $\delta_{\chi}$ and $\delta_{\textup{hom}}$ with increasingly stronger hypotheses. First define a `higher moment minimum degree': for $a\in\I N$, let $\delta^{(a)}(G)$ be the minimum co-degree over all independent sets of size $a$. So $\delta^{(1)}(G)=\delta(G)$ and $\delta^{(2)}(G)=\min\{|N(u)\cap N(v)|:uv\notin E(G)\}$. Then, we can define a higher moment homomorphism threshold as follows. 
$$\delta_{\textup{hom}}^{(a)}(H):=\inf\big\{\alpha\ge 0:\exists~H\textup{-free}~F=F(\alpha,H)\  \textup{s.t.~}\forall~n\textup{-vertex\ }H\textup{-free}~G,~\delta^{(a)}(G)\ge \alpha n\Rightarrow G\xrightarrow{\textup{hom}} F  \big\}.$$  
More generally, for $a,b\in\I{N}$ with $b\ge 2$, let $\hat{\delta}^{(a,b)}(G)$ be the minimum relative $K_{b}$-density in the subgraph induced by the co-neighborhood $N(I)$ of $I$ over all independent sets $I$ of size $a$. In other words, 
$$\hat{\delta}^{(a,b)}(G)=\min\limits_{I:|I|=a}\frac{k_{b}(G[N(I)])}{\binom{|N(I)|}{b}},$$ where $k_b(\cdot)$ counts the number of copies of $K_b$. Define
\begin{align*}
    \delta_{\textup{hom}}^{(a,b)}(H) :=  \inf\big\{\alpha\ge 0: & \exists~H\textup{-free}~\textup{graph}~F=F(\alpha,H)\ 
     \textup{s.t.~} \forall~n\textup{-vertex\ }H\textup{-free}~G,\\&~\delta^{(a)}(G)=\Omega(n),~\hat{\delta}^{(a,b)}(G)\ge \alpha\Rightarrow G\xrightarrow{\textup{hom}} F  \big\}.
\end{align*}
Thus, $\delta_{\textup{hom}}=\delta_{\textup{hom}}^{(1)}$ and~\cref{thm:main1} can be stated as $\delta_{\textup{hom}}^{(2,r-2)}(K_{r})=0$. 

We shall see that~\cref{thm:main1} determines all other variations $\delta_{\textup{hom}}^{(2,b)}(K_{r})$, which take values from generalized Tur\'an densities: $\pi_{s}(K_{t}):=\lim\limits_{n\rightarrow\infty}\frac{\textup{ex}(n,K_{s},K_{t})}{\binom{n}{s}}$, where $s\le t$. Here $\textup{ex}(n,K_{s},K_{t})$ is the maximum number of $K_{s}$ in an $n$-vertex $K_{t}$-free graph. Erd\H{o}s~\cite{1962ErdosPi} determined back in 1962 all such densities $\pi_{s}(K_{t})=\prod_{i=1}^{s-1}\frac{t-1-i}{t-1}$, which are realized by $T_{n,t-1}$.

\begin{cor}\label{cor:GeneralThresholds}
    Let $r,a,b\in\I{N}$ with $r\ge 3$, $a\ge 2$, and $2\le b\le r-2$. Then we have 
    $$\delta_{\textup{hom}}^{(a)}(K_{r})=\frac{r-3}{r-2} \quad \mbox{and} \quad \delta_{\textup{hom}}^{(a,b)}(K_{r})=\pi_{b}(K_{r-2})=\prod_{i=1}^{b-1}\frac{r-3-i}{r-3}.$$
\end{cor}

\medskip
{\bf \noindent Notations.} For a family $\mathcal{H}$ of graphs, we say $G$ is $\mathcal{H}$-free if $G$ does not contain any graph $H\in\mathcal{H}$ as a subgraph. For a vertex $u\in V(G)$, we will use $N_{G}(u)$ (or $N(u)$ if the subscript is clear) to denote the set of neighbors of $u$. For a pair of vertices $u,v\in V(G)$, we use $N_{G}(u,v)$ (or $N(u,v)$ if the subscript is clear) to denote the set of common neighbors of $u$ and $v$, that is, $N_{G}(u,v)=N_{G}(u)\cap N_{G}(v)$. For a subset $T\subseteq V(G)$, we will use $G[T]$ to denote the subgraph induced by $T$. For the sake of clarity of presentation, we omit
floors and ceilings and treat large numbers as integers whenever this does not affect the argument. All logarithms are natural logarithms with base $e$.

\medskip
{\bf \noindent Structure of this paper.} In~\cref{sec:GeometricConsequences}, we will build the connection between graph theory and abstract convexity spaces and prove~\cref{thm:pq-for-C}, which implies~\cref{thm:main0}. In~\cref{sec:Preliminary}, we will discuss the difference between the chromatic and homomorphism threshold problems for cliques and prove~\cref{thm:VC-notfor-hom}. We then prove~\cref{thm:main1,thm:main1-lowerbound} and~\cref{cor:GeneralThresholds} in~\cref{sec:ProofOfMainTheorem}. Concluding remarks are in~\cref{sec:ConcludingRemarks}.

\section{A geometric framework}\label{sec:GeometricConsequences}
In this section, first in~\cref{sec:conv-sp}, we introduce abstract convexity spaces and some useful parameters. Then in~\cref{sec:B}, we define a convexity space from a graph, prove some properties of this space, and establish some correspondence between the graph and this space. With this preparation, we are then in a position to lay out our geometric approach in~\cref{sec:pf-map} and list the main lemmas for proving~\cref{thm:main0}. The main lemmas are proved in~\cref{sec:VC-FracHelly} and~\cref{subsec:bi-ind-matching}. Finally,~\cref{thm:pq-for-C} is proved in~\cref{sec:pq-C}.

\subsection{Abstract convexity space}\label{sec:conv-sp}
A convexity space is a pair $(X,\C C)$ where $\C C\subseteq 2^X$ is a family of subsets satisfying
\begin{itemize}
    \item $\varnothing,X\in \C C$;
    \item $\C C$ is closed under intersection, i.e. for every $\C F\subseteq \C C$, $\bigcap \C F:=\bigcap_{F\in \C F}F\in \C C$.
\end{itemize}
We call sets in $\C C$ \Yemph{convex sets}. The \Yemph{convex hull} of a set of points $Y\subseteq X$ is the intersection of all convex sets in $\C C$ containing $Y$, which is also the minimal convex set containing $Y$. That is, $\M{conv}Y=\bigcap_{Y\subseteq F\in \C C}F$. When the set $X$ is clear from the content, we simply refer $\C C$ as a convexity space. For more on the theory of convexity space and combinatorial convexity, we refer the interested readers to the books of van de Vel~\cite{1993VandeVel} and of B\'ar\'any~\cite{2021BookCombConv}. Two trivial convexity spaces are $\C C=\{\varnothing,X\}$ and $\C C=2^X$. Let $\mathcal{C}^d$ be the family of all convex sets in $\I R^d$. Then $(\I R^d,\mathcal{C}^d)$ is the usual Euclidean convexity space. Here are some non-trivial examples.
\begin{itemize}
    \item \emph{Convex lattice sets}: A set of the form $C\cap \I{Z}^d$ for some convex set $C$ in $\C C^d$ is called a \Yemph{convex lattice set}. Then $(\I Z^d, \C C^d\cap \I{Z}^d)$ is a convexity space.

    \item \emph{Subgroups}: Given a group $G$ with identity $e$, then $(G\setminus \{e\},\{H\setminus\{e\}:H\le G\})$ is a convexity space. For a set $S\subseteq G\setminus \{e\}, {\rm conv}S=\left\langle S \right\rangle \setminus \{e\}$ is a subgroup generated by $S$ (with identity removed).

    \item \emph{Subcubes}. A subset $\C C\subseteq\{0,1\}^n$ is a \emph{subcube} if there exists a set of coordinates $I\subseteq[n]$ and a binary vector $\boldsymbol{v}\in\{0,1\}^I$ such that $\C C$ consists of all vectors in $\{0,1\}^n$ whose projection on $I$ is $\boldsymbol{v}$. Then $\{0,1\}^n$ together with the family of all subcubes forms a convexity space. 
    
    \item \emph{Subtrees}. Given a finite tree $T$ on vertex set $V$, then $V$ together with all its subtrees forms a convexity space.
\end{itemize}

Let us now introduce some invariants which are abstractions of properties of Euclidean convexity space.

\begin{defn}\label{defn:Radon}
The \Yemph{Radon number} of a convexity space $(X,\C C)$, denoted by $r(\C C)$, is the minimum integer $r$ such that for any set of points $Y\subseteq X$ with $|Y|\ge r$, there is a partition $Y=Y_1\cup Y_2$ with $\M{conv}Y_1\cap \M{conv}Y_2\neq\varnothing$.
\end{defn}

Radon's Lemma~\cite{1921Radon} states, in this notation, that the Radon number of the Euclidean convexity space $(\I R^d,\C C^d)$ is at most $d+2$.

\begin{defn}
 The \Yemph{Helly number} of a set system $\C F$, denoted by $h(\C F)$, is the minimum integer $h$ such that in any finite subfamily $\C G\subseteq \C F$, if every $h$-tuple of $\C G$ is intersecting, then $\C G$ is intersecting.
\end{defn}

Thus, Helly's theorem~\cite{1923Helly} implies that $h(\mathcal{C}^d)\le d+1$, which is less than its Radon number. This relation between Radon and Helly numbers holds for all convexity space $(X,\C C)$: $h(\C C)<r(\C C)$, as shown by Levi~\cite{1951Levi}. The gap between Radon number and Helly number, however, could be arbitrarily large. One such example is the space of subcubes in $\{0,1\}^n$, which has Helly number 2, but Radon number $\lfloor{\log_2(n+1)}\rfloor+1$.

We also need the fractional extension of Helly number defined as follows.

\begin{defn}\label{defn:frac-Helly-number}
The \Yemph{fractional Helly number} of a set system $\C F$, denoted by $h^*(\C F)$, is the smallest natural number $k$ such that for every $\alpha>0$, there exists $\beta=\beta(\alpha)>0$ such that the following holds. For every $\{F_{1},F_{2},\ldots,F_{m}\}\subseteq \C F$, if the number of intersecting $k$-tuples $I\in \binom{[m]}{k}$ with $\bigcap\limits_{i\in I}F_{i}\neq\varnothing$ is at least $\alpha\binom{m}{k}$, then $\{F_{1},F_{2},\ldots,F_{m}\}$ contains an intersecting subfamily of size at least $\beta m$.
\end{defn}

We remark that there is no direct relation between the Helly number and its fractional counterpart. For example, for the Euclidean convexity space, $h(\mathcal{C}^d)=h^*(\mathcal{C}^d)=d+1$; for the convex lattice set, $h(\C C^d\cap \I{Z}^d)=2^d$ and $h^*(\C C^d\cap \I{Z}^d)=d+1$~\cite{2003AdvConv,1973Jgeom}; for the family $\C B^d$ of all axis-aligned boxes in $\I R^d$, $h(\C B^d)=2$ and $h^*(\C B^d)=d+1$~\cite{1988Israel}.

The concept of (weak) $\eps$-net is well-studied in computational geometry, combinatorics, and machine learning; it is particularly useful in many algorithmic applications, including range searching and geometric optimization.

\begin{defn}
    Let $X$ be a set and $\C C\subseteq 2^X$ be a family of subsets of $X$. Given a finitely supported probability measure $\mu$ on $X$, a set $N\subseteq X$ is a \emph{weak $\eps$-net for $\C C$ with respect to $\mu$}, if $N \cap C\neq \varnothing$ for every $C \in \mathcal{C} $ with $\mu(C)\geq \varepsilon$. We say that $\C C$ has weak $\eps$-nets of size $m=m(\C C,\eps)$ if there is a weak $\eps$-net for $\C C$ with respect to $\mu$ of size at most $m$ for every finitely supported probability measure $\mu$ on $X$.
\end{defn}

Note that $m(\C C,\eps)$, the size of weak $\eps$-nets for $\C C$, does not depend on the choice of the probability measure $\mu$.

We need the following fractional version of transversal number.

\begin{defn}
A \Yemph{fractional transversal} for a set system $\mathcal{G}$ on ground set $V$ is a function $f$ from $V$ to $[0,1]$ such that for every set $A\in \mathcal{G}$, $\sum_{v\in A}f(v)\ge 1$. The size of $f$ is $\sum_{v\in V}f(v)$ and the \Yemph{fractional transversal number} of $\mathcal{G}$ is the minimum size of a fractional transversal for $\mathcal{G}$, denoted by $\tau^{*}(\mathcal{G})$. 
\end{defn}

In particular, the characteristic function of a transversal is a fractional transversal and so for any $\mathcal{G}$, we have $\tau^{*}(\mathcal{G})\le\tau(\mathcal{G})$. 

\begin{defn}
    A \Yemph{matching} in a set system $\mathcal{H}$ is a collection pairwise disjoint sets in $\mathcal{H}$. The \Yemph{matching number} of $\mathcal{H}$ is size of a largest matching in $\mathcal{H}$, denoted by $\nu(\mathcal{H})$. 
\end{defn}

\begin{defn}
    A \Yemph{fractional matching} in a set system $\mathcal{H}$ on ground set $V$ is a function $g$ from $\C H$ to $[0,1]$ such that for every $v\in V$, $\sum_{v\in A}g(A)\le 1$. The size of $g$ is $\sum_{A\in \C H}g(A)$ and the \Yemph{fractional matching number} of $\mathcal{H}$ is the maximum size of a fractional matching in $\mathcal{H}$, denoted by $\nu^{*}(\mathcal{H})$. 
\end{defn}
Similarly, $\nu(\mathcal{H})\le\nu^{*}(\mathcal{H})$ holds for every set system $\mathcal{H}$. Moreover, LP duality infers that $\nu^{*}(\mathcal{H})=\tau^{*}(\mathcal{H})$. Therefore, for any $\mathcal{H}$, we have $\nu(\mathcal{H})\le\nu^{*}(\mathcal{H})=\tau^{*}(\mathcal{H})\le\tau(\mathcal{H})$.

The \Yemph{Vapnik-Chervonenkis dimension} (\Yemph{VC-dimension} for short) is a parameter that measures the complexity of various combinatorial objects, and plays an important role in statistics, algebraic geometry, learning theory, and model theory.
For a set system $\mathcal{F}\subseteq 2^{X}$, the \Yemph{Vapnik-Chervonenkis dimension} of $\mathcal{F}$, denoted by $\VC(\mathcal{F})$, is the largest integer $d$ for which there exists a subset $S\subseteq X$ with $|S|=d$ such that for every subset $B\subseteq S$, one can find a member $A\in\mathcal{F}$ with $A\cap S=B$. In this case, we say that $S$ is \Yemph{shattered} by $\C F$.

The well-known $\eps$-net theorem of Haussler and Welzl~\cite{1987DCG} (also see the book of Matou\v{s}ek~\cite{2022BookMatousek}) provides an inverse inequality for transversal number and its fractional version in set systems with bounded VC-dimension. 

\begin{theorem}[$\eps$-net theorem]\label{thm:eps-net}
    There exists an absolute constant $C$ such that the following holds. Let $\C F$ be a set system with VC-dimension $d$. Then $\tau(\C F)\le Cd\tau^*(\C F)\log\tau^*(\C F)$.
\end{theorem}

\subsection{The unusual suspect}\label{sec:B}



In this subsection, we construct an abstract convexity space from a graph. Given a set system $\C F$, the \emph{intersection closure} of $\C F$, denoted by $\C F^{\cap}$, is the set system obtained from $\C F$ by taking all possible intersections of subfamilies of $\C F$. Given a graph $G$, let $\MIS(G)$ be the family of all maximal independent sets of $G$. 

We can now define our convexity space $\C C(G)$ from a given graph $G$ as follows:
\begin{itemize}
    \item Let $\mathcal{B}(G):=\{K_{v}:v\in V(G)\}$ be the set system on the ground set $\MIS(G)$ indexed by $V(G)$, where $K_{v}=\{I\in\MIS(G): v\in I\}$ consists of all maximal independent sets of $G$ containing $v$.

    \item Let $\C C(G)=\C B(G)^{\cap}$ be the intersection closure of $\C B(G)$. Then $(X,\C C(G))$ is a convexity space, where $X=\MIS(G)$.
\end{itemize}
Note that $\C B(G)$ and $\C C(G)$ could be a family of multi-sets. For example, if two vertices $u,v$ have the same neighborhood in $G$, then $K_u=K_v$. Equivalently, $\C C(G)$ can be defined as $\C C(G)=\{K_{S}:S\subseteq V(G)\}$, where for each subset $S\subseteq V(G)$, $K_{S}=\{I\in X :S\subseteq I\}$ is the set of all maximal independent sets containing $S$. The \Yemph{dual} of a set system $\C H$ is a set system obtained from $\C H$ by swapping the roles of ground elements and sets in $\C H$. Note that the dual of $\mathcal{B}(G)$ is the set system induced by $\MIS(G)$, which we denote by $\C M(G)=\{I: I\in\MIS(G)\}$.

The convex hull operator of this convexity space $(X,\C C(G))$ can be described as follows. Given a subset $Y=\{I_{1},\ldots,I_{m}\}\subseteq X$, $\M{conv}Y$ is the intersection of all convex sets in $\C C(G)$ containing $Y$, so $\M{conv}Y=K_{I_{1}\cap\cdots\cap I_{m}}$. 

In the rest of this subsection, we establish some correspondence between a graph $G$ and the set system $\C B(G)$, which is summarized in~\cref{tab:table}.

\begin{table}
    \centering
    \begin{tabular}{c|c}
        \toprule
        \textbf{The graph $G$ } & \textbf{Set syetem $\mathcal{B}=\mathcal{B}(G)$} \\
        \midrule
        $uv\in E(G)$ & $K_{u}\cap K_{v}=\emptyset$ \\
        \hline
        (maximal) independent set $I$ & (maximal) intersecting subfamily $\{K_{v}: v\in I\}$ \\
        \hline
        Chromatic number $\chi(G)$ & Transversal number $\tau(\mathcal{B})$ \\
        \hline
        Clique number $w(G)$ & Matching number $\nu(\mathcal{B})$ \\
        \hline
        $K_{r}$-freeness & $(r,2)$-property \\
        \bottomrule
    \end{tabular}
    \caption{Graph terminology vs. geometric terminology}
    \label{tab:table}
\end{table}

We first observe that $G$ is isomorphic to the disjointness graph of $\C B(G)$. Here, the \Yemph{disjointness graph} of a set system $\mathcal{F}$, denoted by $D(\mathcal{F})$, is the graph with vertex set $\mathcal{F}$ and a pair of sets in $\C F$ are adjacent in $D(\mathcal{F})$ if and only if they are disjoint. 
On the other hand, if we start with any set system $\C F$ and consider $\C B(D(\C F))$, we will get a set system whose nerves have the same 1-skeleton. The nerve of a set system is an abstract simplicial complex that records all intersecting subfamilies. In other words, $\C F$ and $\C B(D(\C F))$ have the same pairwise intersections. One can think the two operations $D(\cdot)$ and $\C B(\cdot)$ as dual of each other.

\begin{prop}\label{prop:dual-op}
    Given any graph $G$, we have $G\cong D(\mathcal{B}(G))$. On the other hand, given any set system $\mathcal{F}$, $\C F$ and $\C B(D(\C F))$ have the same pairwise intersections.
\end{prop}
\begin{proof}
    By the definition of $\C B(G)$, $K_u\cap K_v=\varnothing$ if and only if there is no maximal independent set containing both $u$ and $v$, in other words $uv\in E(G)$. Thus, mapping $K_v$ to $v$ for each $v\in V(G)$ is a graph isomorphism between $D(\mathcal{B}(G))$ and $G$.

    For the second part, we need to show that $\C F$ and $\C B(D(\C F))$ have the same pairwise intersections. This amounts to proving that $D(\C F)\cong D(\C B(D(\C F)))$, which follows from the first part.
\end{proof}

A simple but useful fact is that the operation: $G\rightarrow \C C(G)$ always produces a convexity space with Helly number 2.

\begin{prop}\label{prop:HellyNumber2}
    For any graph $G$ with at least one edge, $\C C(G)$ has Helly number $2$.
\end{prop}
\begin{proof}
We need to show that for any subfamily $\C S=\{K_{S_i}: S_i\subseteq V(G)\}\subseteq \C C(G)$, if $\C S$ is pairwise intersecting, then it is in fact intersecting. Similarly to~\cref{prop:dual-op}, it is easy to see that if $K_S$ and $K_T$ intersect, then $S\cup T$ is an independent set. Thus, $\C S$ being pairwise intersecting infers that $S=\bigcup\limits_{S_i\in\C S}S_i$ is an independent set in $G$. Let $I$ be an arbitrary maximal independent set containing $S$, then $I\in\bigcap\C S$ as desired.
\end{proof}

\cref{prop:dual-op} in particular implies that there is a one-to-one correspondence between (maximal) independent sets $I$ in $G$ and subfamilies of $\C B(G)$ that are (maximally) pairwise intersecting; and as $\C B(G)$ has Helly number 2, all these subfamilies are intersecting.

A key correspondence is as follows.

\begin{prop}\label{prop:Chi=Tau}
For any graph $G$, we have $\chi(G)=\tau(\mathcal{B}(G))$.
\end{prop}
\begin{proof}
Observe that a collection of maximal independent sets pierces $\C B(G)=\{K_{u}:u\in V(G)\}$ if and only if their union covers $V(G)$. We first show that $\chi(G)\le \tau(\C B(G))$.
Let $k=\tau(\C B(G))$, then there are $I_{1},I_{2},\ldots,I_{k}\in \MIS(G)$ such that they pierce $\C B(G)$. As $\bigcup_{j\in [k]}I_j=V(G)$, for every vertex $v\in V(G)$, we can color it with the smallest index $j\in[k]$ such that $v\in I_{j}$. This provides a proper $k$-coloring of $G$, and so $\chi(G)\le k$. 

To show $\tau(\mathcal{B}(G))\le\chi(G)$, suppose that $\chi(G)=r$. Then we can partition $V(G)$ into $r$ independent sets $V_{1}\cup V_{2}\cup\cdots\cup V_{r}$. For each $V_j$, $j\in[r]$, let $I_j\supseteq V_j$ be an arbitrary maximal independent set containing it. Then we have $\bigcup_{j\in[r]}I_j=V(G)$ and therefore $I_j$, $j\in[r]$, pierce $\C B(G)$, which implies that $\tau(\C B(G))\le r$. 
\end{proof}

\begin{rmk}\label{rmk:chi-bounded}
Using the correspondence in~\cref{tab:table}, it is not hard to see that~\cref{thm:main0} is equivalent to~\cref{thm:main-rephased}, which is an $(r,2)$-theorem for $\C B(G)$. A class of graphs are $\chi$-bounded if its chromatic number can be bounded by a function of its clique number. The correspondence in~\cref{tab:table} shows that~\cref{thm:main0}, the $(r,2)$-theorem for $\C B(G)$, can be rephrased as the statement that ultra maximal $K_r$-free graphs are $\chi$-bounded. Finally, note that~\cref{thm:pq-for-C} is a strengthening as it is an $(r,2)$-theorem for the convexity space $\C C(G)$, and the $(r,2)$-theorem holds for any family of convex sets in the convexity space, not just those of the family $\C B(G)$.
\end{rmk}

\subsection{Main lemmas and proof of~\cref{thm:main0}}\label{sec:pf-map}
In this subsection, we present our geometric framework, depicted in~\cref{fig:Relation}. By~\cref{prop:Chi=Tau}, our goal is to bound the transversal number of $\C B(G)$. Our approach consists of two sides. 

On the side of convexity space, to bound $\tau(\C B(G))$, by the $\eps$-net theorem,~\cref{thm:eps-net}, it suffices to bound the VC-dimension of $\C B(G)$ and the fractional transversal number $\tau^*(\C B(G))$. The latter can in turn be bounded by the fractional Helly number of $\C B(G)$. To this end, we prove in~\cref{lem:VC-Matching} and~\cref{lem:BoundedHellyNumber} that both the VC-dimension and the fractional Helly number of $\C B(G)$ can be controlled by certain induced matching defined as follows.

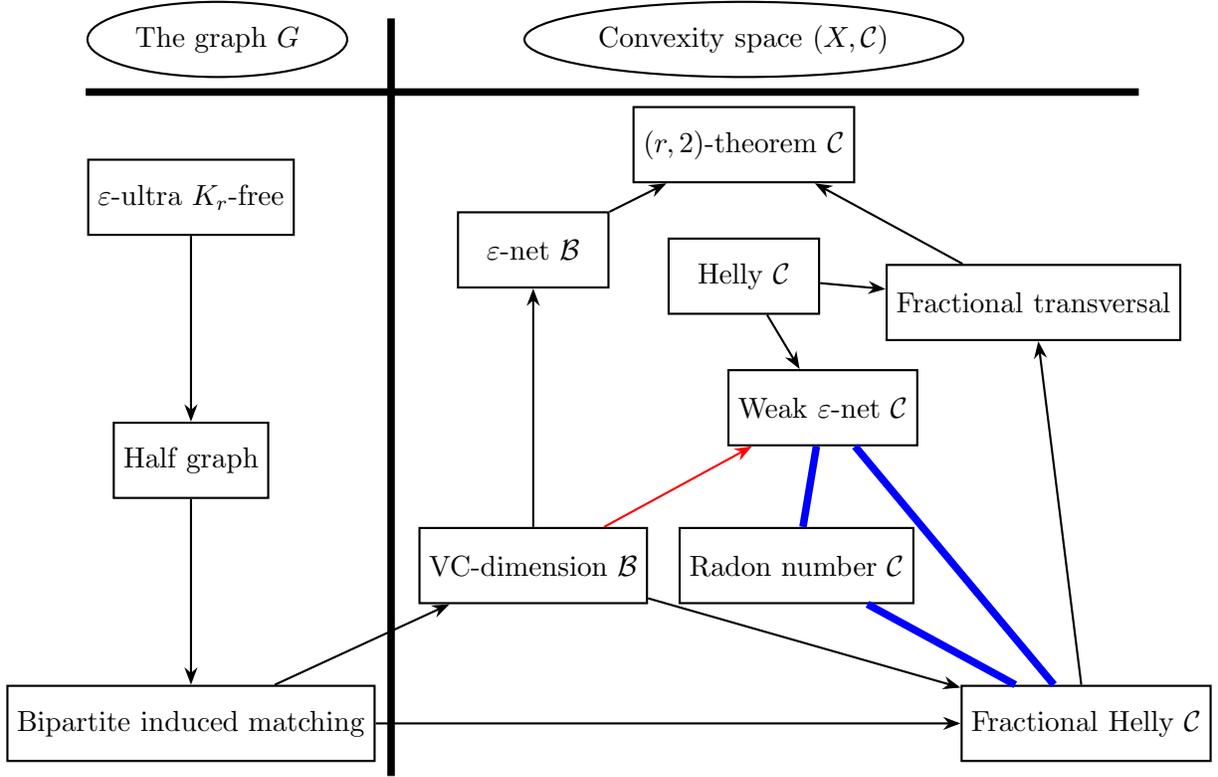
\begin{figure}
    \centering
    \begin{tikzpicture}[>=Stealth, line width=0.8pt,scale=0.7]

\node[draw, rectangle, minimum width=2cm, minimum height=1cm] (A) at (-2,0) {Bipartite induced matching};
\node[draw, rectangle, minimum width=2cm, minimum height=1cm] (B) at (-2,5) {Half graph};
\node[draw, rectangle, minimum width=2cm, minimum height=1cm] (C) at (-2,10) {$\eps$-ultra $K_{r}$-free};
\node[draw, rectangle, minimum width=2cm, minimum height=1cm] (D) at (4.5,3) {VC-dimension $\C B$};
\node[draw, rectangle, minimum width=2cm, minimum height=1cm] (E) at (4.5,9) {$\eps$-net $\C B$};
\node[draw, rectangle, minimum width=2cm, minimum height=1cm] (F) at (9.5,3) {Radon number $\C C$};
\node[draw, rectangle, minimum width=2cm, minimum height=1cm] (G) at (10,6) {Weak $\eps$-net $\C C$};
\node[draw, rectangle, minimum width=2cm, minimum height=1cm] (H) at (8.5,11) {$(r,2)$-theorem $\C C$};
\node[draw, rectangle, minimum width=2cm, minimum height=1cm] (I) at (14,8) {Fractional transversal};
\node[draw, rectangle, minimum width=2cm, minimum height=1cm] (J) at (15,0) {Fractional Helly $\C C$};
\node[draw, rectangle, minimum width=2cm, minimum height=1cm] (K) at (8.5,8.5) {Helly $\C C$};
\node[draw, ellipse, minimum width=1.5cm, minimum height=1cm] (L) at (-1.5,13) {The graph $G$};
\node[draw, ellipse, minimum width=5cm, minimum height=1cm] (M) at (8.5,13) {Convexity space $(X,\C C)$};

\draw[<-] (A) -- (B); 
\draw[->] (C) -- (B); 
\draw[->] (A) -- (D); 
\draw[->] (D) -- (E); 
\draw[->,red] (D) -- (G); 
\draw[->] (E) -- (H); 
\draw[->] (I) -- (H); 
\draw[->] (D) -- (J); 
\draw[->] (K) -- (I); 
\draw[->] (J) -- (I);
\draw[->] (A) -- (J);
\draw[->] (K) -- (G);

\draw[line width=1mm,blue] (F) -- (G); 
\draw[line width=1mm,blue] (G) -- (J); 
\draw[line width=1mm,blue] (F) -- (J); 
\draw[line width=1mm] (1.8,-1) -- (1.8,13.4);
\draw[line width=1mm] (-4,12) -- (16,12);
\end{tikzpicture}
    \caption{The relationship between graphs and convexity spaces}
    \label{fig:Relation}
\end{figure}

A matching $\{u_{i}v_{i}\}_{i\in[t]}$ in a graph $G$ is a \Yemph{bipartite induced matching} of size $t$ if $u_{i}v_{j}\in E(G)$ if and only if $i=j$. We call the size of a largest such matching in $G$ its \Yemph{bipartite induced matching number}, denoted by $\match(G)$. If we further require that both $\{u_{i}\}_{i\in[t]}$ and $\{v_{i}\}_{i\in[t]}$ are independent sets, then we call it an \Yemph{induced matching}.

\begin{lemma}\label{lem:VC-Matching}
    For any graph $G$, if $\VC(\C B(G))\ge 3$, then $\VC(\C B(G))\le \match(G)$.
\end{lemma}

We remark that the bound on the VC-dimension above is optimal. Indeed, consider the graph $G$ with four vertices $x,y,z,w$, where $x,y,z$ form a copy of triangle, and $w$ is an isolated vertex. Then it is easy to check that $\C B(G)$ has VC-dimension 2, however $\match(G)=1$.

\begin{lemma}\label{lem:BoundedHellyNumber}
For any graph $G$, $h^{*}(\mathcal{B}(G))\le \match(G)+1$.
\end{lemma}

We have then to control the bipartite induced matching number on the graph side. Recall that an $n$-vertex $K_{r}$-free graph $G$ is $\varepsilon$-ultra maximal $K_{r}$-free if for every non-adjacent pair of vertices $u,v\in V(G)$, the induced subgraph $G[N(u)\cap N(v)]$ contains $\varepsilon n^{r-2}$ copies of $K_{r-2}$. 

We show that all $\varepsilon$-ultra maximal $K_{r}$-free graphs have bounded bipartite induced matching number.

\begin{theorem}\label{thm:BoundedInducedMatching}
     For $\eps>0$ and $r\ge 3$, every $\eps$-ultra maximal $K_r$-free graph $G$ satisfies 
     $$\match(G)\le (2/\eps)^{2/\eps}.$$
\end{theorem}

We will prove~\cref{thm:BoundedInducedMatching} in two steps. First, we show that in an $\eps$-ultra $K_r$-free graph $G$, there does not exist a large copy of certain half graph, see~\cref{lemma:NoHalfGraph}. From not containing a large half graph, we can then derive that $G$ cannot have a large bipartite induced matching, see~\cref{lemma:ObstructionToIDM}.

Combining~\cref{lem:VC-Matching,lem:BoundedHellyNumber} and~\cref{thm:BoundedInducedMatching}, we get the following.

\begin{cor}\label{cor:ultra-boundedVCandFractHelly}
    Let $\eps>0$, $r\ge 3$ and $G$ be an $\eps$-ultra maximal $K_r$-free graph. Then both the VC-dimension and fractional Helly number of $\C B(G)$ are at most $(2/\eps)^{2/\eps}+1$.
\end{cor}

\begin{rmk}\label{rmk:Matousek}
   We would like to mention a related $(p,q)$-theorem for set systems with bounded VC-dimension due to Matou\v{s}ek~\cite{2004DCG}. It states that if the VC-dimension of the dual of a set system $\C F$ is at most $k-1$, then  any finite $\C G\subseteq \C F$ with $(r,k)$-property, $r\ge k$, satisfies $\tau(\C G)=O_{r,k}(1)$. Notice that even though we do have a bound on the VC-dimension of dual of $\C B(G)$, that is, $\VC(\C M(G))$ (from~\cref{cor:ultra-boundedVCandFractHelly} or~\cref{lem:VCdualBGBiIndMatching}). However, we cannot apply Matou\v{s}ek's result to bound $\tau(\C B(G))$ directly. Indeed, by~\cref{tab:table}, the $K_r$-freeness of $G$ corresponds to $(r,2)$-property of $\C B(G)$. To utilize Matou\v{s}ek's result, we would need to (i) either show the $(r,k)$-property for $\C B(G)$, (ii) or prove that $\C M(G)$ has VC-dimension $1$. Having (i) for any $k>2$ would require a stronger hypothesis that in $G$, there is an independent set of size $k$ in a subgraph induced by any set of $r$ vertices. We cannot hope to have (ii) either: for $\eps$-ultra maximal $K_r$-free graph $G$, $\C M(G)$ could have VC-dimension $2^{\Omega(1/\eps)}$, see e.g. the construction in~\cref{thm:main1} and~\cref{sec:lowerbound-blowup}.
\end{rmk}

To prove~\cref{thm:main0}, we are left to bound the fractional transversal number of $\C B(G)$.

\begin{lemma}\label{lem:UltraToFracTau}
   Let $\eps>0$, $r\ge 3$ and $G$ be an $\eps$-ultra maximal $K_r$-free graph. Then $\tau^{*}(\mathcal{B}(G))=O_{\eps,r}(1)$.
\end{lemma}

Before we prove this result, we need the following.

\begin{theorem}[\cite{2002AKMM}]\label{thm:fractionaltrans}
    Let $r,\ell \ge k\ge 2$ and $\C F$ be a set system with $(r,k)$-property. If the Helly number of $\mathcal{F}$ is $k$ and the fractional Helly number of $\mathcal{F}$ is $\ell$, then $\tau^{*}(\mathcal{F})= O_{r,\ell,k}(1)$.
\end{theorem}

\begin{proof}[Proof of~\cref{lem:UltraToFracTau}]
By the correspondence in~\cref{tab:table}, $\C B(G)$ has $(r,2)$-property as $G$ is $K_r$-free. The Helly number of $\mathcal{B}(G)$ is $2$ by~\cref{prop:HellyNumber2} and the fractional Helly number of $\C B(G)$ is also bounded by~\cref{cor:ultra-boundedVCandFractHelly}. The conclusion then follows from~\cref{thm:fractionaltrans} with $k=2$.
\end{proof}

Finally, combining~\cref{thm:eps-net},~\cref{cor:ultra-boundedVCandFractHelly} and~\cref{lem:UltraToFracTau}, we obtain~\cref{thm:main0}.

\subsection{VC-dimension and fractional Helly number}\label{sec:VC-FracHelly}
In this subsection, we prove~\cref{lem:VC-Matching,lem:BoundedHellyNumber}.

\begin{proof}[Proof of~\cref{lem:VC-Matching}]
    Suppose the VC-dimension of $\mathcal{B}(G)$ is $d\ge 3$, then there is a set $A:=\{I_{1},\ldots,I_{d}\}$ of $d$ maximal independent sets shattered by $\mathcal{B}(G)$. We shall find a bipartite induced matching of size $d$ in $G$. Note that for each $i\in [d]$, there exists some $K_{v_{i}}\in\mathcal{B}(G)$ such that $K_{v_{i}}\cap A=A\setminus \{I_{i}\}$. Then $v_i\notin I_i$, which together with the maximality of $I_i$ implies that there is a vertex $u_{i}\in I_{i}$ such that $v_{i}u_{i}\in E(G)$. Moreover, for all distinct $i,j\in [d]$, $v_{i}\in I_{j}$, implying that $v_{i}u_j\notin E(G)$. That is, $v_i$ is adjacent to $u_j$ if and only if $i=j$. In particular, all $v_i$ are distinct vertices and all $u_i$ are distinct vertices. 

    We now show that for any $i,j\in[d]$, $v_i\neq u_j$. Suppose $v_i=u_j$. Then $i\neq j$, for otherwise $v_j=v_i=u_j$ contradicting $u_jv_j\in E(G)$. Note also that for any distinct $i,j\in[d]$, $K_{v_{i}}\cap K_{v_{j}}\cap A=A\setminus\{I_{i},I_{j}\}\neq\varnothing$ as $d\ge 3$, implying that $u_jv_j=v_{i}v_{j}\notin E(G)$, a contradiction. Thus, $\{u_{i}v_{i}\}_{i\in [d]}$ is a bipartite induced matching of size $d$ as desired.  
\end{proof}

We can also bound the VC-dimension of the dual of $\C B(G)$ by its bipartite induced matching number.

\begin{lemma}\label{lem:VCdualBGBiIndMatching}
    For any graph $G$, if the VC-dimension of the dual of $\mathcal{B}(G)$ is $d\ge 1$, then $d\le\nu_{bi}(G)$.
\end{lemma}
\begin{proof}
    Recall that the dual of $\C B(G)$ is the maximal independent set hypergraph $\C M=\C M(G)$. By assumption, there is a set $D$ of $d$ vertices $v_{1},\ldots,v_{d}$ such that for any subset $D'\subseteq D$, there is a maximal independent set $I_{D'}$ such that $I_{D'}\cap D=D'$. In particular, $D$ is an independent set by considering $I_D$. For each $i\in [d]$, as $v_i\notin I_{D\setminus \{v_i\}}$, the maximality of $I_{D\setminus \{v_i\}}$ implies that there exists some vertex $u_{i}\in I_{D\setminus \{v_i\}}$ such that $v_{i}u_{i}\in E(G)$ and $v_{j}u_{i}\notin E(G)$ for any $j\neq i$. Since $D$ is an independent set, we must have $u_i\notin D$. Moreover, $u_{1},u_{2},\ldots,u_{d}$ are all distinct. Indeed, if $u_j=u_i$ for $j\neq i$, then $v_ju_j=v_ju_i\notin E(G)$, a contradiction. Thus, $\{v_{i}u_{i}\}_{i\in [d]}$ is a bipartite induced matching of size $d$.
\end{proof}

We now bound the fractional Helly number of $\C B(G)$. Matou\v{s}ek~\cite{2004DCG} proved that if the VC-dimension of the dual of a set system $\C F$ is at most $k-1$, then the fractional Helly number of $\C F$ is at most $k$. This, together with~\cref{lem:VCdualBGBiIndMatching}, implies that $h^*(\C B(G))\le \VC(\C M(G))+1\le  \match(G)+1$. Here, we give a different proof using the following result of Holmsen~\cite{2020Holmsen}.

\begin{theorem}[\cite{2020Holmsen}]\label{thm:HolmsenInducedFree}
  For any $k\ge 2$ and $\alpha>0$, there exists $\beta>0$ such that the following holds. Let $G$ be an $n$-vertex graph with at least $\alpha\binom{n}{k}$ independent sets of size $k$. If $G$ does not contain an induced matching of size $k$, then $\alpha(G)\ge\beta n$.
\end{theorem}

\begin{proof}[Proof of~\cref{lem:BoundedHellyNumber}]
Let $\match(G)=k-1$. By the definition of fractional Helly number, we need to show that for every $\alpha >0$, there exists $\beta>0$ such that every family $\mathcal{F}=\{K_{u_{1}},\ldots,K_{u_{m}}\}\subseteq \C B(G)$ with at least $\alpha\binom{m}{k}$ intersecting $k$-subsets must contain an intersecting subfamily of size $\beta m$. By the correspondence in~\cref{tab:table}, this amounts to proving that for every $U=\{u_1,\ldots,u_m\}\subseteq V(G)$, if $G[U]$ has at least $\alpha\binom{m}{k}$ independent sets of size $k$, then $\alpha(G[U])\ge \beta m$. This follows immediately from~\cref{thm:HolmsenInducedFree} as $G$ does not contain an induced matching of size $\match(G)+1=k$.
\end{proof}

\subsection{Bipartite induced matching in $\varepsilon$-ultra maximal $K_{r}$-free}\label{subsec:bi-ind-matching}
In this subsection, we prove~\cref{thm:BoundedInducedMatching}. We need a notion of half graphs. Let $\mathcal{H}_{k}$ be the family of graphs which consists of all of the graphs $H$ with vertices $\{x_{i},y_{i}:~i\in[k]\}$ such that 
\begin{itemize}
    \item $x_{i}y_{i}\notin E(H)$ for all $i\in[k]$; and
    \item $x_{i}y_{j}\in E(H)$ for all $1\le j<i\le k$. 
\end{itemize}
Note that we put no restriction on pairs $\{x_{j},y_{i}\}$ for all $1\le j<i\le k$ and no restriction on $H[\{x_{1},\ldots,x_{k}\}]$ and $H[\{y_{1},\ldots,y_{k}\}]$. We call graphs in $\C H_k$ \Yemph{half graphs}.

\begin{figure}[htbp]
\begin{minipage}[t]{0.48\textwidth}
\centering
\begin{tikzpicture}[>=Stealth, line width=0.8pt, node distance=2cm,scale=0.35]
\node[draw, circle, inner sep=1pt, label={left:$x_{4}$}] (A) at (0,0) {};
\node[draw, circle, inner sep=1pt, label={left:$x_{3}$}] (B) at (0,2) {};
\node[draw, circle, inner sep=1pt, label={left:$x_{2}$}] (C) at (0,4) {};
\node[draw, circle, inner sep=1pt, label={left:$x_{1}$}] (D) at (0,6) {};
\node[draw, circle, inner sep=1pt, label={right:$y_{4}$}] (E) at (3,0) {};
\node[draw, circle, inner sep=1pt, label={right:$y_{3}$}] (F) at (3,2) {};
\node[draw, circle, inner sep=1pt,label={right:$y_{2}$}] (G) at (3,4) {};
\node[draw, circle, inner sep=1pt, label={right:$y_{1}$}] (H) at (3,6) {};

\draw[dashed] (A) -- (E);
\draw[dashed] (B) -- (F);
\draw[dashed] (C) -- (G);
\draw[dashed] (D) -- (H);

\draw (A) -- (F);
\draw (A) -- (G);
\draw (A) -- (H);
\draw (B) -- (G);
\draw (B) -- (H);
\draw (C) -- (H);

\end{tikzpicture}
\caption{Half graph}\label{fig:halfgraph}
\end{minipage}
\begin{minipage}[t]{0.48\textwidth}
\centering
\begin{tikzpicture}[>=stealth, line width=0.8pt, scale=0.35]


\node[draw, circle, inner sep=1pt, label={left:$b_{1}$}] (b1) at (-2.5,6) {};
\node[draw, circle, inner sep=1pt, label={above:$a_{1}$}] (a1) at (0,6) {};
\node[draw, circle, inner sep=1pt, label={above:$c_{1}$}] (c1) at (2.5,6) {};
\node[draw, circle, inner sep=1pt, label={above:$a_{2}$}] (a2) at (0,4) {};
\node[draw, circle, inner sep=1pt, label={right:$a_{\frac{\varepsilon q}{2}+1}$}] (aq2) at (0,2) {};
\node[draw, circle, inner sep=1pt, label={left:$a_{q}$}] (aq) at (0,-1) {};


\draw[red, thick] (b1) -- (a1);

\draw[red, dashed, thick] (b1) -- (a2);
\draw[red, dashed, thick] (b1) -- (aq2);
\draw[red, dashed, thick] (b1) -- (aq);
\draw[red, dashed, thick] (a1) -- (c1);

\draw[red, thick] (c1) -- (a2);
\draw[red, thick] (c1) -- (aq2);

\draw[blue, dashed, thick] (aq2) -- (aq);

\end{tikzpicture}
\caption{\cref{cl:buildHalf}}\label{fig:FindLargeHalf}
\end{minipage}
\end{figure}

\begin{lemma}\label{lemma:NoHalfGraph}
   Let $\eps>0$, $r\ge 3$ and $G$ be an $\eps$-ultra maximal $K_r$-free graph. Then $G$ does not contain any half graph in $\mathcal{H}_{k}$ as a subgraph, where $k=\frac{1}{\eps}+1$.
\end{lemma}
\begin{proof}
Suppose to the contrary that $G$ contains a copy of some half graph $H\in\mathcal{H}_{k}$ with vertex set $\{x_{i},y_{i}:~i\in[k]\}$. As $G$ is $\eps$-ultra maximal $K_r$-free and $x_{i}y_{i}\notin E(G)$ for all $i\in[k]$, the induced subgraph $G[N(x_{i},y_{i})]$ contains at least $\varepsilon n^{r-2}$ many $(r-2)$-cliques. By the pigeonhole principle, there must exist a copy of $K_{r-2}$ lying in at least $\lceil\frac{k\eps n^{r-2}}{\binom{n}{r-2}}\rceil\ge \lceil k\eps\rceil\ge 2$ common neighborhood of distinct pairs, say $G[N(x_{i},y_{i})]$ and $G[N(x_{j},y_{j})]$ with $1\le j<i\le k$. Then, as $x_iy_j\in E(G)$, this copy of $K_{r-2}$ together with $x_{i}$ and $y_{j}$ forms a copy of $K_{r}$, a contradiction.
\end{proof}

\begin{lemma}\label{lemma:ObstructionToIDM}
    Let $\eps>0$, $r\ge 3$ and $G$ be an $\eps$-ultra maximal $K_r$-free graph. If $G$ contains a bipartite induced matching of size $t$, then $G$ contains a copy of some graph $H\in\mathcal{H}_{k}$ as a subgraph, where $k=\log t/\log\frac{2}{\eps}$. 
\end{lemma}
\begin{proof}
 We first establish the following claim, which will be used iteratively to build a large half graph from a large bipartite induced matching.
 
\begin{claim}\label{cl:buildHalf}
    Let $a_1,\ldots,a_q$ and $b_1$ be vertices with $q\ge 2/\eps$, $a_1b_1\in E(G)$ and $a_ib_1\notin E(G)$ for all $2\le i\le q$. Then there exists a vertex $c_1$ such that $a_1c_1\notin E(G)$ and $a_ic_1\in E(G)$ for $2\le i\le \eps q/2+1$. 
\end{claim}
 \begin{poc}
     Consider the set of non-adjacent pairs $\{a_{i},b_{1}\}$, $2\le i\le q$. By assumption, for each such non-adjacent pair there are at least $\varepsilon n^{r-2}$ many $(r-2)$-cliques in $G[N(a_{i},b_{1})]$. Then, by the pigeonhole principle, there exists a copy $K$ of $(r-2)$-clique lying in at least $\frac{(q-1)\varepsilon n^{r-2}}{\binom{n}{r-2}}\ge \eps q/2$ many  common neighborhood of distinct pairs $\{a_i,b_{1}\}$. By relabelling if necessary, we may assume that $K$ lies in $G[N(a_{i},b_{1})]$ for $2\le i\le \eps q/2+1$. Note that, the vertex $a_{1}$ cannot be adjacent to all vertices in $K$ for otherwise $K$ together with $a_{1}$ and $b_{1}$ would form a copy of $K_{r}$. Therefore, we can pick a vertex $c_{1}$ in $K$ such that $a_{1}c_{1}\notin E(G)$ and $a_{i}c_1\in E(G)$ for $2\le i\le \eps q/2+1$ as desired.
 \end{poc}
 
We now build a large half graph. Let $\{x_iz_i\}_{i\in[t]}$ be a bipartite induced matching in $G$. In the first round, applying~\cref{cl:buildHalf} with $(a_i,q,b_1)=(x_i,t,z_1)$, we obtain a vertex $y_1$ such that $x_1y_1\notin E(G)$ and $x_iy_1\in E(G)$ for $2\le i\le \eps t/2+1$. In general, for $2\le j\le k$, in the $j$-th round, we apply~\cref{cl:buildHalf} with $(a_i,q,b_1)=(x_{i+j-1},(\frac{\eps}{2})^{j-1}t,z_j)$, we obtain a vertex $y_j$ such that $x_jy_j\notin E(G)$ and $x_iy_j\in E(G)$ for $j+1\le i\le (\frac{\eps}{2})^{j}t+j$. Since $t\ge (\frac{2}{\eps})^k$, we can indeed invoke~\cref{cl:buildHalf} for $k$ rounds to obtain vertices $y_1,\ldots, y_k$. Note that all $y_i$ vertices are distinct as for any $1\le j<i\le k$, $x_iy_j\in E(G)$ but $x_iy_i\notin E(G)$.
Then, $\{x_i,y_i:~i\in[k]\}$ induces a half graph in $\C H_k$ with $k=\log t/\log\frac{2}{\eps}$ as desired.
\end{proof}

\cref{thm:BoundedInducedMatching} now follows immediately from~\cref{lemma:NoHalfGraph,lemma:ObstructionToIDM}.

\subsection{Beyond $\chi$-boundedness, an $(r,2)$-theorem for $\C C(G)$}\label{sec:pq-C}
In this subsection, we prove~\cref{thm:pq-for-C}, which  strengthens~\cref{thm:main0,thm:main-rephased}. The idea is to utilize weak $\eps$-nets for the convexity space $\C C(G)$ instead of $\eps$-nets for $\C B(G)$. Then we use the equivalence of weak $\eps$-nets theorem (\cref{thm:weak-eps-net-C}), Radon's lemma (\cref{thm:Radon-C}) and fractionally Helly theorem (\cref{thm:frac-Hally-C}) for a convexity space to derive the desired $(r,2)$-theorem for $\C C(G)$.

We shall need the following results for convexity spaces.

\begin{lemma}[Weak $\eps$-net $\Rightarrow$ Radon~\cite{2020DCGWeakNetRadon}]\label{lem:weak-eps-Radon}
Let $(X,\C C)$ be a convexity space. If the Radon number of $\C C$ is greater than $r>0$, then there is a probability measure $\mu$ on $X$ such that every weak $\eps$-net for $\C C$ with respect to $\mu$ has size at least $(1-2\eps)r$.
\end{lemma}

\begin{theorem}[Radon $\Rightarrow$ fractional Helly~\cite{2021IJMHolmsen}]\label{thm:Radon-FracHelly}
  Every convexity space with Radon number $r$ has fractional Helly number at most $r^{r^{\lceil\log_2r\rceil}}$.
\end{theorem}

\begin{theorem}[\cite{2020DCGWeakNetRadon}]\label{thm:WeakEpsRadon}
    Let $X$ be a set and $\mathcal{F}\subseteq 2^X$ be a compact family with VC-dimension $d$ and Helly number $h$. Then for every $\eps>0$, the intersection closure $\C F^{\cap}$ has weak $\varepsilon$-nets of size at most $(120h^{2}/\eps)^{4hd\log(1/\eps)}$.
\end{theorem}

We will use the following consequence of~\cref{thm:WeakEpsRadon}.

\begin{cor}\label{cor:weak-eps-net}
Let $X$ be a set and $\mathcal{F}\subseteq 2^X$ be a compact family with VC-dimension $d$ and Helly number $h$. Then for any finite set system $\C G\subseteq \C F^{\cap}$, we have $\tau(\mathcal{G})\le (120h^{2}\tau^{*}(\C G))^{4hd\log\tau^{*}(\C G)}$.
\end{cor}
\begin{proof}
  Let $t^*=\tau^*(\C G)$ and $f(\cdot)$ be an optimal fractional transversal of $\C G$ assigning each point $p\in X$ with weight $m_p/D$, where $m_{p},D\in\I{N}$. So, $\sum\limits_{p\in X}\frac{m_{p}}{D}=t^*$ and $\sum\limits_{p\in A}\frac{m_{p}}{D}\ge 1$ for each $A\in\mathcal{G}$. Now define a finitely supported probability measure $\mu$ on $X$ by setting $\mu(p)=\frac{m_p/D}{t^*}$. Then, for each $A\in\mathcal{G}$, $\mu(A)=\sum\limits_{p\in A}\frac{m_p/D}{t^*}\ge\frac{1}{t^*}$. Thus, any weak $\frac{1}{t^*}$-net for $\C F^{\cap}$ with respect to $\mu$ is a transversal of $\mathcal{G}$. The desired bound on $\tau(\C G)$ then follows from~\cref{thm:WeakEpsRadon} with $\eps_{\ref{thm:WeakEpsRadon}}=1/t^*$.
\end{proof}

Combining~\cref{thm:WeakEpsRadon},~\cref{cor:weak-eps-net} with~\cref{prop:HellyNumber2} and~\cref{lem:VC-Matching}, we obtain the following weak $\eps$-net theorem for the convexity space $\C C(G)$ derived from a graph $G$.

\begin{theorem}[Weak $\eps$-net for $\C C(G)$]\label{thm:weak-eps-net-C}
    For any graph $G$ and $\eps>0$, the convexity space $\C C(G)$ has weak $\eps$-nets of size at most $(480/\eps)^{8\match(G)\log(1/\eps)}$. In other words, for any set system $\C G\subseteq \C C(G)$, we have
    $$\tau(\C G)\le (480\tau^{*}(\C G))^{8\match(G)\log\tau^{*}(\C G)}.$$
\end{theorem}

Applying~\cref{thm:weak-eps-net-C} and~\cref{lem:weak-eps-Radon} with $\eps_{\ref{thm:weak-eps-net-C}}=\eps_{\ref{lem:weak-eps-Radon}}=1/4$, we infer that the Radon number of the convexity space $\C C(G)$ is bounded by the bipartite induced matching number of $G$.

\begin{theorem}[Radon for $\C C(G)$]\label{thm:Radon-C}
  For any graph $G$, the Radon number of the convexity space $\C C(G)$ is at most $2^{300\match(G)}$.
\end{theorem}

\begin{rmk}
    Let us see the effect of bounded Radon number on the graph $G$ using the correspondence in~\cref{tab:table}. If $r(\C C(G))=r=O(1)$, then for any maximal independent sets $I_1,\ldots, I_r$, there is a partition $[r]=A\cup B$ such that~$K_{\cap_{a\in A} I_a}\cap K_{\cap_{b\in B} I_b}\neq \varnothing$, that is, $e(\cap_{a\in A} I_a,\cap_{b\in B} I_b)=0$. Note that between any two $I_j$ there is at least one edge by maximality. So the convexity space $\C C(G)$ having bounded Radon number reflects in a sense that $\MIS(G)$ has bounded structure.
\end{rmk}

A fractional Helly theorem for $\C C(G)$ now follows from~\cref{thm:Radon-C,thm:Radon-FracHelly}.

\begin{theorem}[Fractional Helly for $\C C(G)$]\label{thm:frac-Hally-C}
  For any graph $G$, the fractional Helly number of the convexity space $\C C(G)$ is at most $2^{2^{10^5\match(G)^2}}$.  
\end{theorem}

Finally, putting everything together, we can prove~\cref{thm:pq-for-C}.

\begin{proof}[Proof of~\cref{thm:pq-for-C}]
As any upper bound on Helly number or fractional Helly number on $\C C(G)$ is inherited by all its subfamilies, by~\cref{prop:HellyNumber2} and~\cref{thm:frac-Hally-C}, we see that $\C G$ satisfies $h(\C G)=2$ and $h^*(\C G)=O_{t}(1)$. Thus, from~\cref{thm:fractionaltrans}, we get that $\tau^*(\C G)=O_{r,t}(1)$. Then~\cref{thm:weak-eps-net-C} infers that $\tau(\C G)=O_{r,t}(1)$. Finally, the `In particular' part follows from~\cref{thm:BoundedInducedMatching}.
\end{proof}

\section{Distinguishing chromatic and homomorphism thresholds}\label{sec:Preliminary} 
Coincidentally, the chromatic and homomorphism thresholds for cliques have the same value. In this section, we prove~\cref{thm:VC-notfor-hom} via a construction to see the difference between these two problems.

Let us first see the role of VC-dimension. A graph $G$ has VC-dimension $d$ if the set system $\mathcal{F}:=\{N_{G}(v)\subseteq V(G):v\in V(G)\}$ induced by the neighborhood of vertices in $G$ has VC-dimension $d$. In this case of triangle, we know that $\delta_{\chi}(K_{3})=\delta_{\textup{hom}}(K_{3})=\frac{1}{3}$. {\L}uczak and Thomass{\'e}~\cite{2010ColoringViaVCDim} later proved that if a triangle-free graph has bounded VC-dimension, then any linear minimum degree $\delta(G)=\Omega(n)$ suffices to force bounded chromatic number (see~\cref{thm:VCToBoundedChromaticNumber} for a short new proof).

In contrast, \cref{thm:VC-notfor-hom} shows that VC-dimension is not sufficient to force a bounded triangle-free homomorphic image. 

We start with the following construction.

\begin{constr}\label{constr:VC-but-no-hom}
    Let $t\ge 3$ and $\Gamma$ be a $t$-vertex maximal triangle-free graph on vertex set $\{z_i\}_{i\in[t]}$.
    
    \begin{itemize}
        \item Let $M_{t,t}$ be the bipartite graph obtained from removing a perfect matching from the complete bipartite graph $K_{t,t}$. That is, $V(M_{t,t})=\{x_i,y_i\}_{i\in[t]}$ and $x_iy_j\in E(M_{t,t})$ if and only if $i\neq j$.

        \item Let $M_{t,t}\vee \Gamma$ be the graph obtained from taking vertex disjoint union of $M_{t,t}$ and $\Gamma$, and adding edges $x_iz_i$, $y_iz_i$ for all $i\in[t]$.

        \item Let $\Gamma[a,b]$ be the graph obtained from $M_{t,t}\vee \Gamma$ by blowing up each vertex in $M_{t,t}$ into $a$ copies and each vertex in $\Gamma$ into $b$ copies.
    \end{itemize}
\end{constr}

\begin{figure}[htbp]
\begin{minipage}[t]{0.48\textwidth}
\centering
       \begin{tikzpicture}[>=Stealth, line width=0.8pt, node distance=2cm,scale=0.3]
\node[draw, circle, inner sep=1pt, label={left:$x_{4}$}] (A1) at (-8,0) {};
\node[draw, circle, inner sep=1pt, label={left:$x_{3}$}] (B1) at (-8,2) {};
\node[draw, circle, inner sep=1pt, label={left:$x_{2}$}] (C1) at (-8,4) {};
\node[draw, circle, inner sep=1pt, label={left:$x_{1}$}] (D1) at (-8,6) {};
\node[draw, circle, inner sep=1pt, label={right:$y_{4}$}] (E1) at (8,0) {};
\node[draw, circle, inner sep=1pt, label={right:$y_{3}$}] (F1) at (8,2) {};
\node[draw, circle, inner sep=1pt,label={right:$y_{2}$}] (G1) at (8,4) {};
\node[draw, circle, inner sep=1pt, label={right:$y_{1}$}] (H1) at (8,6) {};

\node[draw, ellipse, minimum width=5cm, minimum height=2cm, label={below:$\Gamma$}] (L) at (0,-5) {};

\node[draw, circle, inner sep=1pt,label={below:$z_{2}$}] (L3) at (1,-5) {};

\draw[blue] (C1) -- (L3);
\draw[blue] (G1) -- (L3);

\draw[black] (B1) -- (G1);

\draw[dashed,black] (A1) -- (E1);
\draw[dashed,black] (B1) -- (F1);
\draw[dashed,black] (C1) -- (G1);
\draw[dashed,black] (D1) -- (H1);
   \end{tikzpicture}
\caption{$t=4$ in~\cref{constr:VC-but-no-hom}}
     \label{fig:GeneralConstruction}

\end{minipage}
\begin{minipage}[t]{0.48\textwidth}
\centering
\begin{tikzpicture}[>=Stealth, line width=0.8pt, node distance=2cm,scale=0.3]
\node[draw, circle, inner sep=1pt, label={left:$a_{4}$}] (A1) at (-4,0) {};
\node[draw, circle, inner sep=1pt, label={left:$a_{3}$}] (B1) at (-4,2) {};
\node[draw, circle, inner sep=1pt, label={left:$a_{2}$}] (C1) at (-4,4) {};
\node[draw, circle, inner sep=1pt, label={left:$a_{1}$}] (D1) at (-4,6) {};
\node[draw, circle, inner sep=1pt, label={right:$b_{4}$}] (E1) at (4,0) {};
\node[draw, circle, inner sep=1pt, label={right:$b_{3}$}] (F1) at (4,2) {};
\node[draw, circle, inner sep=1pt,label={right:$b_{2}$}] (G1) at (4,4) {};
\node[draw, circle, inner sep=1pt, label={right:$b_{1}$}] (H1) at (4,6) {};

\node[draw, circle, inner sep=1pt, label={left:$c_{1}^{(1)}$}] (A2) at (-8,-3) {};
\node[draw, circle, inner sep=1pt, label={left:$c_{1}^{(2)}$}] (B2) at (-8,-5) {};
\node[draw, circle, inner sep=1pt, label={left:$c_{2}^{(1)}$}] (C2) at (-8,-7) {};
\node[draw, circle, inner sep=1pt, label={left:$c_{2}^{(2)}$}] (D2) at (-8,-9) {};
\node[draw, circle, inner sep=1pt, label={right:$d_{1}^{(1)}$}] (E2) at (8,-3) {};
\node[draw, circle, inner sep=1pt, label={right:$d_{1}^{(2)}$}] (F2) at (8,-5) {};
\node[draw, circle, inner sep=1pt,label={right:$d_{2}^{(1)}$}] (G2) at (8,-7) {};
\node[draw, circle, inner sep=1pt, label={right:$d_{2}^{(2)}$}] (H2) at (8,-9) {};

\draw[dashed,red] (A1) -- (E1);
\draw[dashed,red] (B1) -- (F1);
\draw[dashed,red] (C1) -- (G1);
\draw[dashed,red] (D1) -- (H1);

\draw[orange] (D1) -- (A2);
\draw[orange] (D1) -- (B2);
\draw[orange] (H1) -- (A2);
\draw[orange] (H1) -- (B2);

\draw[green] (C1) -- (C2);
\draw[green] (C1) -- (D2);
\draw[green] (G1) -- (C2);
\draw[green] (G1) -- (D2);

\draw[cyan] (B1) -- (E2);
\draw[cyan] (B1) -- (F2);
\draw[cyan] (F1) -- (E2);
\draw[cyan] (F1) -- (F2);

\draw[yellow] (A1) -- (G2);
\draw[yellow] (A1) -- (H2);
\draw[yellow] (E1) -- (G2);
\draw[yellow] (E1) -- (H2);

\draw[blue] (A1) -- (F1);
\draw[blue] (A1) -- (G1);
\draw[blue] (A1) -- (H1);
\draw[blue] (B1) -- (G1);
\draw[blue] (B1) -- (E1);
\draw[blue] (B1) -- (H1);
\draw[blue] (C1) -- (H1);
\draw[blue] (C1) -- (F1);
\draw[blue] (C1) -- (E1);
\draw[blue] (D1) -- (E1);
\draw[blue] (D1) -- (F1);
\draw[blue] (D1) -- (G1);

\draw[red] (A2) -- (F2);
\draw[red] (A2) -- (G2);
\draw[red] (A2) -- (H2);
\draw[red] (A2) -- (E2);
\draw[red] (B2) -- (G2);
\draw[red] (B2) -- (E2);
\draw[red] (B2) -- (H2);
\draw[red] (B2) -- (F2);
\draw[red] (C2) -- (H2);
\draw[red] (C2) -- (F2);
\draw[red] (C2) -- (E2);
\draw[red] (C2) -- (G2);
\draw[red] (D2) -- (H2);
\draw[red] (D2) -- (E2);
\draw[red] (D2) -- (F2);
\draw[red] (D2) -- (G2);
\end{tikzpicture}
\caption{$K_{2,2}[1,2]$}
     \label{fig:1/6Example}
\end{minipage}
\end{figure}

It is easy to check that the graph $\Gamma[a,b]$ is maximal triangle-free with $(2a+b)t$ vertices and minimum degree $\min\{(t-1)a+b,~2a+\delta(\Gamma)b\}$. Furthermore, the VC-dimension of $\Gamma[a,b]$ depends only on $\Gamma$. Thus, for any $\Gamma$ with sufficiently large order $t$, bounded VC-dimension and minimum degree $\delta(\Gamma)=\Omega(t)$ and $a,b=O(1)$, $G=\Gamma[a,b]$ is a maximal triangle-free graph with minimum degree $\Omega(|G|)$ and bounded VC-dimension.

Next, we identify the obstruction for a graph to have a small triangle-free homomorphic image.

\begin{lemma}\label{lem:obstruction-hom}
    If a graph $G$ contains $t$ vertices $x_1,\ldots,x_t$ such that between any two distinct $x_i,x_j$ there is an induced 4-vertex path $P_4$, then $G$ has no triangle-free homomorphic image of size smaller than $t$.
\end{lemma}
\begin{proof}
Let $F$ be a triangle-free homomorphic image of $G$ with $\varphi: G\xrightarrow{\textup{hom}} F$. It suffices to show that every vertex $x_i$, $i\in[t]$, must be mapped to a distinct vertex in $F$. Suppose there are two distinct $x_{i}$, $x_{j}$ such that $\varphi(x_{i})=\varphi(x_{j})$. By assumption, there exist vertices $y,z$ such that $x_iyzx_j$ is an induced $P_4$. Then $\varphi(x_i)$, $\varphi(y)$ and $\varphi(z)$ form a triangle in $F$, a contradiction.     
\end{proof}

\begin{prop}\label{prop:nonhomo}
    The graph $\Gamma[a,b]$ in~\cref{constr:VC-but-no-hom} does not have any triangle-free homomorphic image of size smaller than $t$.
\end{prop}
\begin{proof}
By construction, between any two distinct $x_i,x_j$ in $\Gamma[a,b]$, $x_iz_iy_ix_j$ is an induced $P_4$. The conclusion then follows from~\cref{lem:obstruction-hom}.
\end{proof}

\subsection{Proof of~\cref{thm:VC-notfor-hom}}
We shall prove that $G=K_{\frac{n}{8},\frac{n}{8}}[1,2]$ has the desired properties. Let $V(G)=A\cup B\cup C\cup D$, where $A=\{a_{i}\}_{i\in[n/4]}$, $B=\{b_{i}\}_{i\in[n/4]}$, $C=\{c_{i}^{(1)},c_{i}^{(2)}\}_{i\in[n/8]}$ and $D:=\{d_{i}^{(1)},d_{i}^{(2)}\}_{i\in[n/8]}$ are independent sets, and $A\cup B$ induces $M_{n/4,n/4}$ with $a_ib_i\notin E(G)$, $C\cup D$ induces $K_{n/4,n/4}$ and $a_{i}c_{i}^{(k)},b_{i}c_{i}^{(k)}, a_{i+n/8}d_{i}^{(k)},b_{i+n/8}d_{i}^{(k)}\in E(G)$ for all $i\in [n/8]$ and $k\in\{1,2\}$.

By construction, the minimum degree of $G$ is at least $\frac{n}{4}$. \cref{thm:VC-notfor-hom} follows from~\cref{prop:nonhomo} (with $t=\frac{n}{4}$) and the following result.

\begin{prop}\label{prop:BoundedVCDim}
    The VC-dimension of $G=K_{\frac{n}{8},\frac{n}{8}}[1,2]$ is at most $3$.
\end{prop}
\begin{proof}
 Let $S\subseteq V(G)$ be a nonempty subset shattered by $\mathcal{F}:=\{N_{G}(v)\subseteq V(G):v\in V(G)\}$. Then, for every subset $S'\subseteq S$, there is a vertex $v_{S'}$ such that $N_{G}(v_{S'})\cap S=S'$. As $G$ is triangle-free, for every vertex $v\in V(G)$, $N_G(v)$ is an independent set; in particular, $S\subseteq N_{G}(v_{S})$ is an independent set. Then, $S\cap (A\cup B)$ is either $\{a_i,b_i\}$ for some $i\in [\frac{n}{4}]$, or a subset of $A$ or $B$. Similarly, $S\cap (C\cup D)$ is a subset of $C$ or $D$.
 
Suppose that $S\cap (A\cup B)=\{a_i,b_i\}$ for some $i\in [\frac{n}{4}]$. Since $S\cap C$ and $S\cap D$ cannot both be nonempty, we can assume without loss of generality that $S\cap C\neq\emptyset$. If $|S\cap C|=1$, then we are done. Otherwise, if $S\cap C$ contains at least two vertices, say $c,c'\in C$, then, by our construction of $G$, $v_{\{a_i,b_i,c\}}$ can only be in $D$. But then we would also have $c'\in N_G(v_{\{a_i,b_i,c\}})\cap S$, a contradiction. So, $S$ has size at most $3$ in this case. 

Next, by symmetry, suppose that $S\cap (A\cup B)\subseteq A$ and $S\cap (C\cup D)\subseteq C$. Suppose that $S\cap (C\cup D)$ contains at least three vertices, say $c_1,c_2,c_3\in C$, then $N_G(c_i)\neq N_G(c_j)$ for any distinct $i,j\in[3]$, as otherwise there is no choice for $v_{\{c_i\}}$. Consequently, there is no $k\in[\frac{n}{8}]$ so that $\{c_1,c_2\}=\{c_k^{(1)},c_k^{(2)}\}$, which implies that $v_{\{c_1,c_2\}}$ can only be in $D$. However, we would also have $c_3\in N_G(v_{\{c_1,c_2\}})\cap S$, again a contradiction. 
Therefore, $|S\cap (C\cup D)|\le 2$. If $|S\cap (C\cup D)|= 2$, similarly $v_S$ must be in $D$. As by assumption $S\cap (A\cup B)\subseteq A$ and every vertex in $D$ has degree 1 in $A$, we have $|S|\le 3$.

We may then assume $|S\cap (C\cup D)|\le 1$. Suppose that $|S\cap (C\cup D)|= 1$. If $S\cap (A\cup B)\subseteq A$ contains three vertices, say $a_1,a_2,a_3\in A$, then as every vertex in $D$ has degree $1$ to $A$, for each $i\in[3]$, $v_{S\setminus\{a_i\}}$ must lie in $B$. But $S\setminus\{a_i\}$ contains a vertex in $C$ that has degree $1$ in $B$, implying that there is at most one vertex in $B$ that could be $v_{S\setminus\{a_i\}}$ for some $i\in [3]$, a contradiction. Thus, in this case $|S\cap (A\cup B)|\le 2$ and $|S|\le 3$.

Lastly, we assume $S\cap (C\cup D)=\varnothing$. If $|S\cap (A\cup B)|\ge 4$, then there is no vertex adjacent to exactly two vertices in $S\cap(A\cup B)\subseteq A$, a contradiction. Hence $|S|=|S\cap (A\cup B)|\le 3$, finishing the proof.
\end{proof}

We remark that $G=K_{\frac{n}{8},\frac{n}{8}}[1,2]$ is triangle-free and has large minimum degree, but $G$ is not $\varepsilon$-ultra maximal $K_{3}$-free as $|N_G(a_{i},b_{i})|=2$. This also shows that the ultra $K_{r}$-freeness is more important to determine the structure of homomorphic images of dense $K_{r}$-free graphs. 

\section{Sharp blow-up phenomenon}\label{sec:ProofOfMainTheorem} 

\subsection{Proof of Theorem~\ref{thm:main1}}\label{subsec:FormalProof}
The main tool used in the proof of \cref{thm:main1} is the packing lemma of Haussler~\cite{1995PackingLemma}. To state it, we need some notations. In a set system $\C F\subseteq 2^{V}$, we say that $v$ \Yemph{splits} $F_{1},F_{2}\in\C F$ if $v$ belongs to exactly one of $F_{1}$ and $F_{2}$, i.e.~$v\in F_1\triangle F_2$. A subfamily $\mathcal{X}\subseteq\mathcal{F}$ is \Yemph{$s$-separated} if for every pair in $\C X$ is split by at least $s$
elements in $V$. Haussler~\cite{1995PackingLemma} proved that the size of an $s$-separated subfamily in a set system can be bounded using its VC-dimension.

\begin{lemma}[\cite{1995PackingLemma}]\label{lemma:Packing}
     Let $\mathcal{F}\subseteq 2^{V}$ be a set system with VC-dimension $d$. If $\mathcal{X}\subseteq \mathcal{F}$ is $s$-separated, then $|\mathcal{X}|\le e(d+1)\big(\frac{2e|\mathcal{F}|}{s}\big)^{d}$.
\end{lemma}

\begin{proof}[Proof of Theorem~\ref{thm:main1}]
Let $M_{t,t}$ be the graph obtained by removing a perfect matching from  $K_{t,t}$. Let $\mathcal{M}_{t,t}$ be the family of graphs obtained by adding every possible set of edges to the two vertex parts of $M_{t,t}$; that is, $\mathcal{M}_{t,t}$ consists of all graphs $F$ with vertices $\{a_{i},b_{i}\}_{i\in[t]}$ such that $a_{i}b_{j}\in E(F)$ if and only if $i\neq j$ for all $i,j\in[t]$, where we put no restriction on $F[\{a_1,\ldots,a_t\}]$ and $F[\{b_1,\ldots,b_t\}]$. It is clear that $\mathcal{M}_{t,t}$ is a subfamily of $\mathcal{H}_{t}$ defined in~\cref{subsec:bi-ind-matching}. Setting $t:=\frac{1}{\eps}+1$, we infer from~\cref{lemma:NoHalfGraph} that $G$ is $\mathcal{M}_{t,t}$-free. Combining $K_{r}$-freeness and $\mathcal{M}_{t,t}$-freeness, we can bound the VC-dimension of $G$.

\begin{claim}\label{claim:BOundedVC}
The VC-dimension of $G$ is at most $t+r-4$.
\end{claim}
\begin{poc}
    Suppose for contradiction that the VC-dimension of $G$ is at least $t+r-3$. Then, there is a $(t+r-3)$-set $A=\{a_1,\ldots,a_{t+r-3}\}$ such that for each subset $S\subseteq A$, there exists a vertex $v_S\in V(G)$ with $N(v_S)\cap A=S$. In particular, letting $b_i=v_{A\setminus\{a_i\}}$ for each $i\in[t+r-3]$, we have $N(b_{i})\cap A=A\setminus\{a_i\}$. Let $B=\{b_{1},\ldots,b_{t+r-3}\}$.
    
    As $G$ is $K_{r}$-free, it is not hard to see that $|A\cap B|\le r-3$. Indeed, suppose otherwise that $|A\cap B|\ge r-2$. As for each $i\in [t+r-3]$, $b_i$ is adjacent to all vertices in $A$ except for $a_i$, we have that $a_i\neq b_j$ for every distinct $i,j$. Therefore, for any vertex $c\in A\cap B$, we must have $c=a_i=b_i$ for some $i\in [t+r-3]$. Assume without loss of generality that $a_i=b_i$ for $i\in [r-2]$. Then $v_A,a_1,\ldots,a_{r-2}$ together with any vertex in $A\setminus\{a_1,\ldots,a_{r-2}\}$ form a copy of $K_r$, a contradiction.

    As $|A|=|B|=t+r-3$ and $|A\cap B|\le r-3$, by the construction of $A$ and $B$, we can find $2t$ distinct vertices $\{a'_i\}_{i\in[t]}\subseteq A\setminus B$ and $\{b'_i\}_{i\in[t]}\subseteq B\setminus A$ such that $\{a'_i,b'_i\}_{i\in[t]}$ induces a copy of some graph in $\mathcal{M}_{t,t}$, a contradiction. 
\end{poc}

Let $\mathcal{F}:=\{N(v):v\in V(G)\}$ and $s=\frac{\varepsilon |\mathcal{F}|}{10}$. Let $\mathcal{X}:=\{N(v_{1}),N(v_{2}),\ldots,N(v_{m})\}$ be a maximal $(s+1)$-separated subfamily of $\C F$. By~\cref{lemma:Packing} and~\cref{claim:BOundedVC}, $m\le e(d+1)(\frac{2e|\mathcal{F}|}{s+1})^{d}=(\frac{1}{\eps})^{O(\frac{1}{\eps}+r)}$. We then partition $V(G)=V_{1}\cup V_{2}\cup\cdots\cup V_{m}$ as follows. For every $v\in V(G)$, let $v\in V_i$ where $i\in[m]$ is the smallest index such that $|N(v)\triangle N(v_{i})|\le s$; such an index $i$ exists by the maximality of $\mathcal{X}$. Note that by the construction of our partition and the triangle inequality, every pair of vertices $u,v$ in the same part $V_i$ satisfies $|N(u)\triangle N(v)|\le|N(u)\triangle N(v_i)|+|N(v)\triangle N(v_i)|\le 2s$.

\begin{claim}\label{claim:cherry}
    For any $i\in [m]$, $v,w\in V_{i}$, there is no $u\in V(G)$ with $uv\in E(G)$ but $uw\notin E(G)$.
\end{claim}
\begin{poc}
   By assumption, $G[N(u,w)]$ contains $\varepsilon n^{r-2}$ copies of $K_{r-2}$. Since $v,w$ lies in the same part $V_i$, $|N(w)\setminus N(v)|\le 2s$. Thus, the number of $(r-2)$-cliques in $G[N(u,w)]$ with at least one vertex lying in $N(w)\setminus N(v)$ is at most $2s n^{r-3}<\frac{\varepsilon n^{r-2}}{2}$, and so there is a copy of $K_{r-2}$ in $G[N(u,v)]$, contradicting $G$ being $K_r$-free.
\end{poc}

For two disjoint subsets $A,B\subseteq V(G)$, we say that $A$ and $B$ are \Yemph{complete} to each other if every vertex of $A$ is adjacent to every vertex of $B$, and \Yemph{anti-complete} to each other if no vertex of $A$ is adjacent to any vertex of $B$.

\begin{claim}\label{claim:CompleteOrEmpty}
    For each pair of distinct $i,j\in [m]$, $V_{i}$ and $V_{j}$ are either complete, or anti-complete.
\end{claim}
\begin{poc}
If $|V_{i}|=|V_{j}|=1$, then we are done. Suppose that $V_{i}$ and $V_{j}$ are neither complete, nor anti-complete. Then, without loss of generality, there exist $v,w\in V_i$ and $u\in V_j$ such that $uv\in E(G)$ and $uw\notin E(G)$, contradicting~\cref{claim:cherry}. 
\end{poc}

\begin{claim}\label{claim:2}
    For each $i\in [m]$, if $|V_{i}|\ge r$, then $V_{i}$ is an independent set.
\end{claim}
\begin{poc}
   Since $|V_{i}|\ge r$, $G[V_{i}]$ cannot be a clique. Suppose that $V_{i}$ is not an independent set, then there exists a triple $v,w,u\in V_{i}$ such that $uv\in E(G)$ and $uw\notin E(G)$, again contradicting~\cref{claim:cherry}. 
\end{poc}
Finally, if any $V_{i}$ has size smaller than $r$, further partition $V_i$ into $|V_{i}|$ singletons. \cref{claim:CompleteOrEmpty} and~\cref{claim:2} imply that $G=F[\cdot]$, where clearly $F$ is maximal $K_{r}$-free and $|F|\le (r-1)m=(\frac{1}{\eps})^{O(\frac{1}{\eps}+r)}$.
\end{proof}

\subsection{Proof of~\cref{thm:main1-lowerbound}}\label{sec:lowerbound-blowup}

We first prove that the theorem holds for the case $r=3$.

\begin{lemma}\label{lemma:TriangleFreeSharp}
    There exists an $n$-vertex triangle-free graph $G$ such that for any pair of non-adjacent vertices $u,v$, $|N_{G}(u,v)|\ge\eps n$, but $G$ has no triangle-free homomorphic image of size smaller than $2^{\frac{1}{8\varepsilon}}$.
\end{lemma}
\begin{proof}
   Let $d=\lfloor\frac{1}{8\varepsilon}-\frac{1}{2}\rfloor$ and note that $\varepsilon\le \frac{1}{8d+4}$. 
   Let $H$ be a graph on vertex set $D\cup Q$, where $D=\{a_{i}^{(0)},a_{i}^{(1)}\}_{i\in[d]}$ and $Q=\{0,1\}^d$. Clearly, $H$ has $2d+2^{d}$ vertices. For each $\boldsymbol{u}\in Q$, write $\bar{\boldsymbol{u}}=\boldsymbol{1}-\boldsymbol{u}$, where $\boldsymbol{1}$ is the all one vector. The edge set $E(H)$ is defined as follows:
   \begin{itemize}
       \item $a_{i}^{(0)}a_{i}^{(1)}\in E(H)$ for each $i\in [d]$;
       \item $\boldsymbol{u}\bar{\boldsymbol{u}}\in E(H)$ for each $\boldsymbol{u}\in Q$;
       \item each $\boldsymbol{u}=(u_{1},u_{2},\ldots,u_{d})\in Q$ is adjacent to $d$ vertices $\{a_{1}^{(u_{1})},\ldots,a_{d}^{(u_{d})}\}\subseteq D$.
   \end{itemize}
   Let us observe some properties of $H$. It is easy to see that the graph $H$ is maximal triangle-free. 
\begin{claim}\label{claim:Twin-free}
     In $H$, the set $Q$ contains $2^{d-1}$ vertices between any pair of which there is an induced $P_4$. In particular, $H$ has no triangle-free homomorphic image of size smaller than $2^{d-1}$.
\end{claim}
\begin{poc}
    By construction, the subgraph of $H$ induced on $Q$ is a perfect matching $M$ of size $2^{d-1}$. Let $X\subseteq Q$ be an independent set of $2^{d-1}$ vertices containing exactly one endpoint of an edge in $M$. Then, for any two distinct vertices $\boldsymbol{u}=(u_1,\ldots,u_d),\boldsymbol{v}=(v_1,\ldots,v_d)\in X$, they must differ at some coordinate, say $u_i=1$ and $v_i=0$ for some $i\in[d]$. Then $\boldsymbol{u}a_i^{(1)}a_i^{(0)}\boldsymbol{v}$ is an induced $P_4$. The rest then follows from~\cref{lem:obstruction-hom}.
\end{poc}
    
Next, let $G$ be a graph obtained from $H$ by blowing up each vertex $a_i^{(j)}$ in $D$ to a set $A_{i}^{(j)}$ of $2^{d}$ identical copies. Let $V(G)=P\cup Q$, where $P=\bigcup_{i\in[d]}(A_{i}^{(0)}\cup A_{i}^{(1)})$. The graph $G$ satisfies the required co-degree condition.
\begin{claim}
    For each non-adjacent pair of vertices $u,v$ in $G$, $|N_{G}(u,v)|\ge 2^{d-2}$.
\end{claim}
\begin{poc}
    Consider the following cases.
    \begin{Case}
        \item If both of $u,v$ belong to $Q$, then write them as $\boldsymbol{u}$ and $\boldsymbol{v}$. Since $\boldsymbol{u}\boldsymbol{v}\notin E(G)$, we have $\boldsymbol{u}\neq \bar{\boldsymbol{v}}$. Therefore, there exists some coordinate $j\in [d]$ such that $u_{j}=v_{j}$, which implies that $A_{j}^{(u_j)}\subseteq N_G(\boldsymbol{u},\boldsymbol{v})$, and therefore $|N_{G}(\boldsymbol{u},\boldsymbol{v})|\ge 2^{d}$.

    \item If $\boldsymbol{u}\in Q$ and $v\in P$. Assume without loss of generality that $v\in A_{i}^{(0)}$, then $\boldsymbol{u}v\notin E(G)$ implies that $u_{i}\neq 0$. Then $A_{i}^{(1)}\subseteq N_{G}(\boldsymbol{u},v)$, and therefore $|N_{G}(\boldsymbol{u},v)|\ge 2^{d}$. 
    
    \item If $u,v\in P$ and they are copies of the same vertex in $D$, say $u,v\in A_{j}^{(0)}$ for some $j\in [d]$, then $A_{j}^{(1)}\subseteq N_{G}(u,v)$, and so $|N_{G}(u,v)|\ge 2^{d}$.
    \item If $u,v\in P$ but they are copies of distinct vertices in $D$. Suppose that $u\in A_{i}^{(x)}$ and $v\in A_{j}^{(y)}$ for distinct $i,j\in [d]$ and some $x,y\in\{0,1\}$, then $\{\boldsymbol{w}\in Q:w_{i}=x\textup{~and~}w_{j}=y\}\subseteq N_{G}(u,v)$, which implies that $|N_{G}(u,v)|\ge 2^{d-2}$.
    \end{Case}
    The proof of the claim is complete.
\end{poc}

Thus, $G$ is a triangle-free graph with $n=(2d+1)2^{d}$ vertices such that any non-adjacent pair of vertices has $2^{d-2}=\frac{n}{8d+4}\ge\varepsilon n$ common neighbors. Moreover, as $G$ is a blow-up of $H$, by~\cref{claim:Twin-free}, $G$ has no triangle-free homomorphic image of size smaller than $2^{d-1}\le 2^{\frac{1}{8\varepsilon}}$ as desired. 
\end{proof}

For larger $r\ge 4$, let us take an $(r-2)$-partite Tur\'an graph $T_{(r-2)n,r-2}$ and add a copy of $G$ into one of its partite set. It is easy to check that the resulting graph is $\alpha$-ultra maximal $K_r$-free, where $\alpha=\frac{\eps}{r^r}$ and has no $K_r$-free homomorphic image of size $2^{\frac{1}{8\eps}}=2^{\frac{1}{8\alpha r^r}}$.

Addressing~\cref{rmk:Matousek}, we note that the induced subgraph $G[Q]$ is an induced matching of size $2^{d-1}$ with parts say $L\cup R$. Then $L$ is shattered by $\C M(G)$ and so $\C M(G)$ has VC-dimension at least $2^{d-1}=2^{\Omega(\frac{1}{\eps})}$.

\subsection{Proof of~\cref{cor:GeneralThresholds}}
We first show that the minimum degree condition is a stronger condition than the ultra $K_r$-freeness condition. The converse is obviously not true by considering a blow-up of a large maximal $K_r$-free graph.
\begin{prop}\label{prop:MinDegreeToUltra}
   Let $\varepsilon>0$ and $G$ be an $n$-vertex maximal $K_{r}$-free graph with $\delta(G)\ge (\frac{2r-5}{2r-3}+\varepsilon)n$, then $G$ is $\varepsilon^{r-2}$-ultra maximal $K_{r}$-free.
\end{prop}
\begin{proof}
    Let $S\subseteq V$ be a set of vertices and $N(S)$ be the set of common neighbors of vertices in $S$. Note that for every $S\subseteq V(G)$, $|N(S)|\ge |S|\cdot\delta(G)-(|S|-1)n$.
    Let $u,v$ be a non-adjacent pair of vertices in $G$. Then for arbitrary $r-3$ vertices $w_{1},w_{2},\ldots,w_{r-3}\in V(G)\setminus\{u,v\}$, we have
    \begin{equation*}
        |N(\{w_{1},w_{2},\ldots,w_{r-3}\})|\ge (r-3)\delta(G)-(r-4)n\ge \bigg(\frac{3}{2r-3}+(r-3)\varepsilon\bigg)n.
    \end{equation*}
Moreover, it was proved in~\cite[Proposition~2.1]{2020CPCProb} that $N(u,v)\ge r\delta(G)-(r-2)n\ge (\frac{2r-6}{2r-3}+r\varepsilon)n$.

    Therefore, the common neighborhood of $u,v$ and $w_{1},w_{2},\ldots,w_{r-3}$ satisfies
    \begin{equation*}
        |N(\{u,v,w_{1},w_{2},\ldots,w_{r-3}\})|\ge |N(\{w_{1},w_{2},\ldots,w_{r-3}\})|-(n-N(u,v))\ge (2r-3)\varepsilon n\ge \eps n.
    \end{equation*}
    Therefore, for each fixed non-adjacent pair of vertices $u,v$, we can greedily pick $r-2$ vertices $z_{1},z_{2},\ldots,z_{r-2}$, where $z_i$ lies in the common neighborhood of $\{u,v,z_1,\ldots,z_{i-1}\}$ and the number of choices of each $z_{i}$ is at least $\varepsilon n$. Thus, $G$ is $\varepsilon^{r-2}$-ultra maximal $K_{r}$-free.
\end{proof}

\begin{prop}\label{prop:codeg-edge}
    For an $n$-vertex graph $G$, if $\delta^{(2)}(G)\ge cn$, then $\hat{\delta}^{(2,2)}(G)\ge 2-1/c$.
\end{prop}
\begin{proof}
   Pick arbitrary $uv\notin E(G)$ with $d(u,v)=\delta^{(2)}(G)\ge cn$. Let $G'=G[N(u,v)]$ and $n'=|G'|\ge cn$. It suffices to show that $\delta(G')\ge (2-1/c)n'$. Let $w\in V(G')$, then 
   $$d_{G'}(w)\ge (\delta(G)-2)+n'-(n-2)\ge \delta^{(2)}(G)+n'-n= 2n'-n.$$
   Thus, $\frac{\delta(G')}{n'}\ge \frac{2n'-n}{n'}\ge 2-1/c$ as desired.
\end{proof}

\begin{proof}[Proof of~\cref{cor:GeneralThresholds}]
For the lower bound, take the $(r-2)$-partite Tur\'an graph and add to one of its parts a triangle-free graph with large chromatic number, say $k$. Let the resulting graph be $H$. Obviously $H$ has no homomorphic image of size at most $k$ and one can easily check that $\delta^{(a)}(H)=\frac{r-3}{r-2}\cdot n$ and $\hat{\delta}^{(a,b)}(H)=\pi_b(K_{r-2})$, proving the desired lower bound.

For the upper bound, as the condition gets stronger when $a$ increases, we may assume that $a=2$. We first prove $\delta^{(2,b)}_{\textup{hom}}(K_r)\le \pi_b(K_{r-2})$ for each $2\le b\le r-2$. Let $G$ be an $n$-vertex $K_r$-free graph with $\delta^{(2)}(G)=\Omega(n)$ and $\hat{\delta}^{(a,b)}(G)\ge \pi_b(K_{r-2})+\eps$. For any pair of non-adjacent vertices $u,v$, as $G[N(u,v)]$ has $K_b$-density larger than $\pi_b(K_{r-2})$, by a supersaturation version of Erd\H{o}s's result~\cite{1962ErdosPi}, $G[N(u,v)]$ has positive $K_{r-2}$-density. Since $|N(u,v)|\ge \delta^{(2)}(G)=\Omega(n)$, we get that $G$ is $\Omega(1)$-ultra maximal $K_r$-free and the conclusion follows from~\cref{thm:main1} applies.

It remains to prove $\delta^{(2)}_{\textup{hom}}(K_r)\le \frac{r-3}{r-2}$. If an $n$-vertex graph $G$ has $\delta^{(2)}(G)\ge (\frac{r-3}{r-2}+\eps)n$, then by~\cref{prop:codeg-edge}, $\hat{\delta}^{(2,2)}(G)\ge \frac{r-4}{r-3}+\eps>\pi_2(K_{r-2})$. The result then follows from the upper bound we proved: $\delta^{(2,2)}_{\textup{hom}}(K_r)\le \pi_2(K_{r-2})$.
\end{proof}

\section{Concluding remarks}\label{sec:ConcludingRemarks}

\subsection{Thresholds with bounded VC-dimension}\label{sec:VCThreshold}
Based on the proof of~\cref{thm:main1}, we give a new short proof of the following result of {\L}uczak and Thomass{\'e}~\cite{2010ColoringViaVCDim}.
\begin{theorem}\label{thm:VCToBoundedChromaticNumber}
  Let $c>0$ and $d\in\I{N}$. Let $G$ be an $n$-vertex triangle-free graph with VC-dimension $d$ and minimum degree $\delta(G)\ge cn$, then $\chi(G)\le e(d+1)\big(\frac{2e}{3c}\big)^{d}$.
\end{theorem}
\begin{proof}
    Let $\C F=\{N_{G}(v):v\in V(G)\}$. Applying~\cref{lemma:Packing} with $s=\frac{c|\mathcal{F}|}{3}$, we get a partition $V(G)=V_{1}\cup V_{2}\cup\cdots\cup V_{m}$ with $m\le e(d+1)\big(\frac{2e}{3c}\big)^{d}$ such that for any pair of vertices $u,v$ from the same part $V_{i}$, $|N_{G}(u)\triangle N_{G}(v)|\le 2s$. This implies that $u,v$ have at least $cn-2s\ge\frac{cn}{3}$ common neighbors and thus $uv\notin E(G)$ as $G$ is triangle-free. Hence, each $V_{i}$ is an independent set and $\chi(G)\le m$.
\end{proof}

Motivated by~\cref{constr:VC-but-no-hom}, we can define the \Yemph{Bounded-VC chromatic/homomorphism thresholds} for graph $H$ as follows:
\begin{align*}
    \delta_{\chi}^{\textup{VC}}(H) :=  \inf\big\{\alpha\ge 0: &~\forall d\in\I N, \exists~C=C(\alpha,H,d)\ 
     \textup{s.t.~} \forall~n\textup{-vertex\ }H\textup{-free}~G,\\&~\VC(G)\le d,~\delta(G)\ge \alpha n\Rightarrow \chi(G)\le C \big\},
\end{align*}
and
\begin{align*}
    \delta_{\textup{hom}}^{\textup{VC}}(H) :=  \inf\big\{\alpha\ge 0: &~\forall d\in\I N, \exists~H\textup{-free}~\textup{graph}~F=F(\alpha,H,d)\ 
     \textup{s.t.~} \forall~n\textup{-vertex\ }H\textup{-free}~G,\\&~\VC(G)\le d,~\delta(G)\ge \alpha n\Rightarrow G\xrightarrow{\textup{hom}} F  \big\}.
\end{align*}
Then~\cref{thm:VCToBoundedChromaticNumber} implies that $\delta_{\chi}^{\textup{VC}}(K_3)=0$, while~\cref{thm:VC-notfor-hom} implies that $\delta_{\textup{hom}}^{\textup{VC}}(K_{3})\ge\frac{1}{4}$. We have later noticed that using the Andr\'asfai graph one can get the optimal bound $\delta_{\textup{hom}}^{\textup{VC}}(K_{3})\ge\frac{1}{3}$. By adding the Andr\'asfai graph into one part of a complete $(r-2)$-partite graph, we can obtain that $\delta^{\textup{VC}}_{\textup{hom}}(K_{r})\ge \frac{2r-5}{2r-3}$. The following summarizes what we know:
$$\delta_{\textup{hom}}^{(2,b)}(K_{r})<\cdots <\delta_{\textup{hom}}^{(2,2)}(K_{r})< \delta_{\textup{hom}}^{(2)}(K_{r})=\frac{r-3}{r-2}=\delta_{\chi}^{\textup{VC}}(K_{r})<\delta_{\textup{hom}}^{\textup{VC}}(K_{r})=\frac{2r-5}{2r-3}.$$
It would be interesting to determine $\delta^{\textup{VC}}_{\textup{hom}}(H)$ for general graph $H$.

\subsection{Related results on large blow-ups}
A well-known result of Nikiforov~\cite{2008BLMSNikiforov} states that any graph with positive $K_r$-density contains a copy of $\Omega(\log{n})$-size blow-up of $K_r$. Considering random graphs, the logarithmic order is best possible. Some recent work studies natural conditions that guarantee much larger blow-ups, see e.g.~\cite{2019EuJCPachTomon}. \cref{thm:main1} can also be viewed as a result of this type, implying that $\eps$-ultra maximal $K_{r+1}$-freeness guarantees a linear size $\Omega(n)$, the largest possible, blow-up of $K_{r}$. 

\begin{cor}
    Let $G$ be an $\eps$-ultra maximal $K_{r+1}$-free graph, then $G$ contains a copy of $K_{r}[Cn]$, where $C=C(r,\varepsilon)$.
\end{cor}

Our results are also related to an old result of Pach~\cite{1981PachStructure}, where he studied the property $I_k$ which is equivalent to $\delta^{(k)}(G)\ge 1$. Answering a question of Erd\H{o}s and Fajtlowics, Pach~\cite{1981PachStructure} proved that a triangle-free graph has the property that every of its independent set has a common neighbor if and only if it is a blow-up of the Cayley graph on $\I Z/(3k-1)\I Z$ generated by a maximum size sum-free set. Brandt~\cite{2002Brandt} later extended Pach's result, showing that a triangle-free graph $G$ having this property is also equivalent to the property of not containing any induced $C_{6}$.

\subsection{Refining partitions from regularity lemma}
As ultra maximal $K_r$-free graphs have bounded VC-dimension, by a stronger version of regularity lemma~\cite{2019DCGVC}, it admits an almost homogeneous partition, i.e.~the density between most of of pairs are either $o(1)$ or $1-o(1)$. By~\cref{thm:main1}, we know that such graphs actually have a homogeneous partition. It is an interesting problem on its own to find natural sufficient conditions that would allow one to boost an almost homogeneous partition to a homogeneous one. We refer the interested readers to an alternative proof of~\cref{thm:main1} in~\cite{2024Second}, where we provide one such example.

By Lemma~\ref{lemma:NoHalfGraph}, ultra maximal $K_{r}$-free graph does not have a large half graph as subgraph. This is also related to a \Yemph{stable regularity lemma} due to Malliaris and Shelah~\cite{2014TAMSStable} where they study graphs without a variation of the half graph considered here. We refer the interested readers to the recent note~\cite{2021NoteSatble}.

\section*{Acknowledgement}
The first author would like to thank Andreas Holmsen for many stimulating conversations. The authors would like to thank Yuval Wigderson for drawing attention to~\cite{2011Unpubilished}. The authors would like to thank Mingyuan Rong for pointing out that the construction of Andr\'{a}sfai graphs which provides the sharp lower bound for bounded VC-dimension homomorphism thresholds of cliques in~\cref{sec:VCThreshold}.


\bibliographystyle{amsplain}
\bibliography{CodegHomoThres}

\end{document}